\documentclass[12pt]{article}

\usepackage{amsfonts,amsmath,amsthm,amssymb}
\usepackage{bm}
\usepackage[colorlinks, linkcolor=blue, citecolor=blue, urlcolor=blue, backref=page]{hyperref}
\usepackage[round]{natbib}
\usepackage{geometry}
\usepackage[T1]{fontenc}    
\usepackage{microtype}

\usepackage{enumitem}

\newtheorem{theorem}{Theorem}[section]

\newtheorem{lemma}[theorem]{Lemma}

\theoremstyle{definition}
\newtheorem{definition}[theorem]{Definition}

\newtheorem{remark}[theorem]{\textbf{Remark}}
\numberwithin{equation}{section}

\newcommand{\C}{{\mathbb C}}

\newcommand{\R}{{\mathbb R}}
\newcommand{\N}{{\mathbb N}}

\newcommand{\Acal}{{\mathcal A}}

\newcommand{\Dcal}{{\mathcal D}}
\newcommand{\Ecal}{{\mathcal E}}
\newcommand{\Fcal}{{\mathcal F}}

\newcommand{\Hcal}{{\mathcal H}}

\newcommand{\Lcal}{{\mathcal L}}

\newcommand{\Rcal}{{\mathcal R}}
\newcommand{\Scal}{{\mathcal S}}

\newcommand{\real}{{\rm Re}}
\renewcommand{\i}{{\rm i}}

\DeclareMathOperator{\Res}{Res}
\DeclareMathOperator{\supp}{supp}

\DeclareMathOperator{\rank}{rank}

\newcommand{\crit}{{\textnormal{crit}}}
\newcommand{\supcr}{{\textnormal{sup}}}
\newcommand{\subcr}{{\textnormal{sub}}}
\newcommand{\vol}{\textnormal{vol}}

\begin{document}

\title{Langevin dynamics along the zero set of real-analytic potentials}
\author{
Martin Larsson\footnote{Department of Mathematical Sciences, Carnegie Mellon University, \texttt{larsson@cmu.edu}}
}
\maketitle

\begin{abstract}
We consider the Langevin diffusion $dX_t = - \beta \nabla V(X_t) dt + \sqrt{2} dB_t$ for a general nonnegative real-analytic potential $V$ and a large parameter $\beta$.
In the large-$\beta$ limit the process is confined to the zero set of $V$, assuming that it starts there.
We derive a candidate limiting evolution on the zero set.
To do so, the zero set is partitioned into strata according to a measure of local codimension known as the local learning coefficient and its multiplicity.
It is then shown that the Dirichlet form associated with $X$ converges in a certain sense to a hierarchy of Dirichlet forms corresponding to a stochastic evolution on the strata.
This evolution is strongly biased toward higher-dimensional, or ``more singular'', strata.
This result is motivated by a question from Watanabe's singular learning theory regarding the learning dynamics of overparameterized statistical models and the generalization puzzle in deep learning.
The result suggests a mechanism for the observation that stochastic gradient methods tend to be biased toward singular solutions that generalize well.
\end{abstract}

\newpage

\section{Introduction and main results}

Let $V$ be a nonnegative real-analytic function on $\R^d$ whose zero set $V^{-1}(0)$ is nonempty and not the whole space. 
Consider the Langevin equation
\begin{equation} \label{eq:original-SDE}
dX_t = - \beta \nabla V(X_t) dt + \sqrt{2} \, dB_t
\end{equation}
starting from $X_0 = x_0 \in \R^d$, where $B$ is a $d$-dimensional standard Brownian motion and $\beta > 0$ a parameter.
This equation has a pathwise unique global strong solution, and we denote this solution by $X^\beta$.\footnote{Strong existence and pathwise uniqueness are immediate from smoothness of $V$. To see that no explosions occur, one can note that $X^\beta$ is the canonical Brownian motion on the manifold $\R^d$ with the weighted volume $e^{-\beta V(x)} dx$, which is stochastically complete thanks to \citet[Theorem~3.13]{MR2218016}. This means by definition that $X^\beta$ is nonexplosive.}
Our goal is to understand the behavior of the family of processes $\{X^\beta\}$ in the large-$\beta$ limit, assuming that the (fixed) starting point $x_0$ belongs to the zero set $V^{-1}(0)$.

This problem is motivated by a question in singular learning theory \citep{MR2554932,MR3839272} regarding the learning dynamics of overparameterized statistical models.
In that context, $\R^d$ is the parameter space of an overparameterized model, $V(x)$ is the empirical loss at the parameter vector $x \in \R^d$ evaluated on a fixed dataset, and \eqref{eq:original-SDE} models the evolution of a learning algorithm such as stochastic gradient Langevin dynamics (SGLD) or stochastic gradient descent (SGD) with isotropic noise and learning rate $1/\beta$.
A fuller discussion is postponed to Section~\ref{sec:slt-connection} below.

For large $\beta$, $X^\beta$ is unable to move far from the zero set, and any limit process will be confined to it.
The purpose of this paper is to describe the limiting evolution on the zero set.
Our results are not enough to deduce weak convergence to the limit, a question which this paper leaves open.
Nonetheless, just describing a candidate limit without further assumptions on $V$ is a rather general problem, so to build intuition and to appreciate some of the difficulties involved, it is useful to look (heuristically) at an example.

Let $d=2$ and consider the potential
\[
V(x) = x_1^{2k_1} x_2^{2k_2}
\]
for some positive integers $k_1,k_2$.
The zero set is the union $\{x_1 = 0\} \cup \{x_2 = 0\}$ of the coordinate axes, and the dynamics for finite $\beta$ are
\begin{equation} \label{eq:x-beta-dynamics-example}
dX^\beta = - \beta V(X^\beta) \begin{pmatrix} 2k_1 / X^\beta_1 \\ 2k_2 / X^\beta_2 \end{pmatrix} dt + \sqrt{2} \, dB,
\end{equation}
where we drop the time index for brevity.
First, it can be shown using on argument based on It\^o's formula that $\beta V(X^\beta) = O(1)$ on bounded time intervals, uniformly in $\beta$.
This establishes a transverse length scale: suppose for concreteness that we start on the $x_1$-axis away from the origin.
Then, at least initially, it follows that $X^\beta_2 = O(\beta^{-1/(2k_2)})$, which is close to zero.
Thus $X^\beta_2$ is subject to a much stronger drift than $X^\beta_1$ and acts as a ``fast'' variable. The normalized process $Y = \beta^{1/(2k_2)} X^\beta_2$ behaves like a one-dimensional Langevin diffusion, running in time scaled by $\beta^{1/k_2}$, under the potential $W(y) = (X^\beta_1)^{2k_1} y^{2k_2}$ where the value of the ``slow'' variable $X^\beta_1$ is frozen,
\[
dY = - W'(Y) \beta^{1/k_2}dt  + \sqrt{2 \beta^{1/k_2}} dB_2.
\]
For any nonzero frozen value of $X^\beta_1$, this diffusion has a stationary density proportional to $e^{-W(y)}$, and the stationary mean of $\beta V(X^\beta) = W(Y)$ is $1/(2k_2)$, independently of $X^\beta_1$ and $k_1$.
This acts as the effective value of $\beta V(X^\beta)$ in the dynamics of $X^\beta_1$ in \eqref{eq:x-beta-dynamics-example} and yields, at least heuristically for large $\beta$, that
\[
dX^\beta_1 \approx - \frac{k_1}{k_2} \frac{1}{X^\beta_1} dt + \sqrt{2} dB_1.
\]
We thus expect that if the starting point lies on the punctured $x_1$-axis, the limiting process will evolve as a (two-sided) $\delta_1$-dimensional Bessel process along the $x_1$-axis with $\delta_1 = 1 - k_1/k_2$.
By symmetry, if we start on the punctured $x_2$-axis, we obtain a two-sided $\delta_2$-dimensional Bessel process along the $x_2$-axis with $\delta_2 = 1 - k_2/k_1$.
Note that both $\delta_1$ and $\delta_2$ belong to $(-\infty,1)$.
Thus the origin is attained in both cases.
Moreover, a zero or negative dimension forces absorption, whereas this is not the case for dimensions in $(0,1)$.
This suggests two separate regimes:
\begin{itemize}
\item $k_1 > k_2$ (the case $k_1 < k_2$ follows by relabeling): starting on the $x_1$-axis, the process evolves as a Bessel process of dimension $\delta_1 < 0$ until it hits the $x_2$-axis, gets trapped there, and continues forever as a (two-sided) Bessel process of dimension $\delta_2 \in (0,1)$ along the $x_2$-axis. 

\item $k_1 = k_2$: starting away from the origin, the process evolves along its coordinate axis as a zero-dimensional Bessel process until it hits the origin, where it gets trapped and the evolution stops.
\end{itemize}

Even in this simple case a rigorous description of the limiting diffusion is nontrivial.
Bessel processes of dimension $\delta \in (-\infty, 1)$ fail to be semimartingales, and cannot be described as the solution of an It\^o SDE.
Nonetheless, they are Markov processes and can be described through their generator.
However, the domain of the generator involves derivative conditions at zero, which in the two-sided situation above must be tailored to capture the reflection symmetry about the origin.
It is not clear how this would extend to the general case, as the zero set of a general nonnegative real-analytic function can have very complicated geometry.

Instead, to leverage real analyticity we shift focus from the generator to its associated Dirichlet form.
Returning to the general case \eqref{eq:original-SDE},
$X^\beta$ is a symmetric diffusion with respect to the weight $e^{-\beta V}$, and its Dirichlet form $\Ecal_\beta$ is given on smooth compactly supported functions by
\begin{equation} \label{eq:dirichlet-form-beta}
\Ecal_\beta(f,g) = \int_{\R^d} \nabla f(x) \cdot \nabla g(x) \, e^{-\beta V(x)} dx.
\end{equation}
The Dirichlet form characterizes the process $X^\beta$ and is related to the generator $\Lcal^\beta$ through the identity $\Ecal_\beta(f,g) = \langle f, -\Lcal^\beta g\rangle_\beta$, where
$\langle \cdot, \cdot \rangle_\beta$ is the $e^{-\beta V}$-weighted $L^2$ inner product.
In the example above, the two Dirichlet forms
\begin{equation} \label{eq:intro-dirichlet-bessel}
\int_{-\infty}^\infty f'(x_1) g'(x_1) |x_1|^{-k_1/k_2} dx_1 \quad \text{and} \quad \int_{-\infty}^\infty f'(x_2) g'(x_2) |x_2|^{-k_2/k_1} dx_2
\end{equation}
describe the evolution of the limit process along the $x_1$-axis and $x_2$-axis, respectively.
Importantly, the reflection symmetry about the origin is captured by the symmetry of the weights, not by any derivative conditions on $f$ and $g$ at zero.
In the general case this is a crucial advantage.

The right-hand side of \eqref{eq:dirichlet-form-beta} is a Laplace integral whose large-$\beta$ asymptotics for real-analytic $V$ is a classical topic which plays a central role in singular learning theory.
Using techniques from this literature, coupled with methods from subanalytic geometry and geometric measure theory, we show that the forms $\Ecal_\beta$ converge after rescaling and on suitable function classes to a hierarchy of Dirichlet forms $\Ecal_{\lambda,m}$ that correspond to a stochastic evolution on $V^{-1}(0)$.

To explain this, and to establish notation, consider a bounded open domain $D$ that meets $V^{-1}(0)$. We assume $D$ is a subanalytic set, for instance a ball or a cube.
Its (nonempty) intersection with the zero set is denoted by
\[
Z = \{x \in D \colon V(x) = 0\}.
\]
To each point $\bar x \in Z$ one can associate a positive rational number $\lambda$, known as the \emph{local learning coefficient} (LLC), and a nonnegative integer $m$, known as the (local) \emph{multiplicity}; see \citet{MR2554932,MR3839272,lau2024locallearningcoefficientsingularityaware}.
The pair $(\lambda,m)$, which we refer to as the \emph{singularity type} of $\bar x$, governs the local volume asymptotics of the sublevel sets of $V$ at $\bar x$: for any sufficiently small ball $B$ centered at $\bar x$, one has
\begin{equation} \label{eq:volume-asymptotics}
\vol(B \cap \{V \le \varepsilon\}) \sim c\, \varepsilon^\lambda (-\log \varepsilon)^{m-1}, \quad \varepsilon \to 0,
\end{equation}
where $\vol(\cdot)$ refers to Lebesgue measure and $c > 0$ is a constant.
The LLC can be interpreted as (half) the codimension of the zero set, locally near $\bar x$; see also \citet{wei2023deep}.

Within any compact set, such as $\overline D$, only finitely many singularity types occur; this is the main reason we restrict attention to $D$. The singularity types may then be totally ordered according to the decay rates in \eqref{eq:volume-asymptotics}:
\begin{equation} \label{eq:ordering}
(\lambda,m) < (\lambda',m') \quad \Leftrightarrow \quad \lambda < \lambda' \text{ or } (\lambda=\lambda' \text{ and } m > m').
\end{equation}
The smaller the singularity type in this ordering, the slower the volume decay in \eqref{eq:volume-asymptotics}.
We then stratify the zero set $Z$ by singularity type,
\begin{equation} \label{eq:Z-lambda-m-def}
Z_{\lambda,m} = \{x \in Z \colon x \text{ has singularity type } (\lambda,m)\}.
\end{equation}
The sets $Z_{\lambda,m}$ are also ordered by \eqref{eq:ordering}, so that ``deep'' or ``more singular'' strata (corresponding to slow volume decay or large local dimension) come before ``shallow'' or ``less singular'' strata (corresponding to faster volume decay or smaller local dimension).
The following statement summarizes the key points of this paper's main results, Theorem~\ref{thm:laplace-asymptotics} and Theorem~\ref{thm:limiting-dirichlet-form}.

\begin{theorem} \label{thm:main}
For each singularity type $(\lambda,m)$, let $\Dcal_{\lambda,m}$ be the set of all $f \in C^1_c(D)$ whose gradient is tangent to $Z_{\lambda,m}$ and whose support does not meet any deeper strata.
Then for any $f,g \in \Dcal_{\lambda,m}$ one has $c_\beta \Ecal_\beta(f,g) \to \Ecal_{\lambda,m}(f,g)$ as $\beta \to \infty$, where
\[
c_\beta = \beta^\lambda (\log \beta)^{1-m}
\]
and
\[
\Ecal_{\lambda,m}(f,g) = \int_{Z_{\lambda,m}} \nabla f \cdot \nabla g \, d\mu_{\lambda,m}
\]
for a Radon measure $\mu_{\lambda,m}$ on $Z_{\lambda,m}$. The bilinear form $\Ecal_{\lambda,m}$ with domain $\Dcal_{\lambda,m}$ is closable in $L^2(Z_{\lambda,m}, \mu_{\lambda,m})$, and the closure is a strongly local regular Dirichlet form.
\end{theorem}

The tangency condition on the gradients of $f \in \Dcal_{\lambda,m}$ requires clarification.
It is shown in Lemma~\ref{lem:Z-lambda-m} that $Z_{\lambda,m}$ is a subanalytic set. As such it admits a Whitney stratification, which is a certain partition into smooth manifolds called strata.
That $\nabla f$ is tangent to $Z_{\lambda,m}$ then simply means that its value at any $x \in Z_{\lambda,m}$ belongs to the tangent space at $x$ of the stratum containing $x$.
Note that this condition depends not only on the set $Z_{\lambda,m}$, but also on the stratification.
The proof of closability of $(\Ecal_{\lambda,m},\Dcal_{\lambda,m})$ relies on finding a stratification with additional properties; see Theorem~\ref{thm:limiting-dirichlet-form} and in particular Lemma~\ref{lem:whitney-strat-for-z-m-lambda}.
That the closure is a regular Dirichlet form further requires that $\Dcal_{\lambda,m}$ is dense in $C_c(Z_{\lambda,m})$, which is shown in Lemma~\ref{lem:c-strat}.

Because the limiting form $(\Ecal_{\lambda,m},\Dcal_{\lambda,m})$ closes to a strongly local regular Dirichlet form, it corresponds to a strong Markov process $X^{\lambda,m}$ with continuous trajectories in the one-point compactification of $Z_{\lambda,m}$; see \citet[Theorems~4.5.3 and~7.2.1]{MR2778606}.
Such trajectories can only leave $Z_{\lambda,m}$ by converging to the frontier $\overline{Z_{\lambda,m}} \setminus Z_{\lambda,m}$.
It turns out that points in the frontier that also belong to $D$ must lie in deeper strata.
Thus $X^{\lambda,m}$ evolves in $Z_{\lambda,m}$ until it either leaves $D$, or hits a deeper stratum $Z_{\lambda',m'}$ inside $D$. In the latter case, at least heuristically, the evolution can continue, now governed by the form $\Ecal_{\lambda',m'}$.
This suggests that any limit process of $\{X^\beta\}$ descends through progressively more singular strata, terminating in the deepest stratum that can be reached from the initial point $x_0$ by continuous motion through $V^{-1}(0)$.

The nature of the dynamics of $X^{\lambda,m}$ can in principle be read off the measure $\mu_{\lambda,m}$.
This is facilitated by Theorem~\ref{thm:laplace-asymptotics}, which analyzes the asymptotics for the Laplace integral in \eqref{eq:dirichlet-form-beta} or, more generally, integrals of the form
\[
\int_D f(x) e^{-\beta V(x)}  dx.
\]
That this is related to volume asymptotics like \eqref{eq:volume-asymptotics} is immediate from Karamata's Tauberian theorem. Furthermore, the specific decay rates are well-known from, e.g., singular learning theory, where they are derived using Hironaka's resolution of singularities \citep{MR199184,MR256156}.
Our proof of Theorem~\ref{thm:laplace-asymptotics} repeats some of these derivations, but, in doing so, obtains more detailed information about the limit.
Specifically, it is shown that $\mu_{\lambda,m}$ is the pushforward of a density $\rho_{\lambda,m}$ which is given in local ``resolved coordinates'' $y = (y_1,\ldots,y_d)$ by a strictly positive real-analytic multiple of
\begin{equation} \label{eq:intro-rho-density}
\prod_{i = m+1}^d |y_i|^{-2k_i \lambda + h_i} dy_i,
\end{equation}
viewed as a density on the codimension-$m$ subspace $E = \bigcap_{i=1}^m \{y_i = 0\}$.
Here $k_i > 0$ and $h_i \ge 0$ are certain integers associated with $V$.
In view of \eqref{eq:intro-rho-density}, the density $\rho_{\lambda,m}$ can blow up at points $y \in E$ with $y_j = 0$ for some $j > m$, and may or may not be integrable there.
Integrability fails if and only if $\lambda \ge (h_i + 1)/(2k_i)$, and such points correspond to deeper strata.
The blowup rates are related to the relative values of the LLCs; see Remark~\ref{rem:blowup} for more details.

In terms of the stochastic dynamics, such nonintegrable blowup corresponds, at least heuristically, to an exploding drift toward $\{y_j=0\}$ which is strong enough to force absorption there.
This is most clearly seen in the two-dimensional example above, where the limiting Dirichlet forms are given by \eqref{eq:intro-dirichlet-bessel}.
This is consistent with \eqref{eq:intro-rho-density}.
Indeed, if $k_1 > k_2$, the LLCs can be computed as $\lambda_1 = 1/(2k_1)$ and $\lambda_2 = 1/(2k_2)$, both with multiplicity $m_1=m_2=1$, corresponding to the strata $\{x_1=0\}$ and $\{x_2 = 0\} \setminus \{0\}$.
For $k_1 = k_2 = k$, one has $\lambda_1 = \lambda_2 = 1/(2k)$ with multiplicities $m_1=2$ and $m_2=1$, corresponding to the strata $\{0\}$ and $(\{x_1=0\}\cup\{x_2=0\}) \setminus \{0\}$.
The parameters $h_i$ in \eqref{eq:intro-rho-density} are zero in this example.

To conclude this introduction we elaborate on the connection with singular learning theory and generalization in deep learning; see Section~\ref{sec:slt-connection} below.
Then, in Section~\ref{sec:sing-types}, we state a version of Hironaka's resolution of singularities and review the definition and properties of the LLC, multiplicity, and stratification of $Z$ by singularity type.
Section~\ref{sec:laplace-asymptotics} derives the Laplace asymptotics, in particular the form of the measures $\mu_{\lambda,m}$.
Section~\ref{sec:dirichlet-limit} analyzes the limiting (pre-)Dirichlet forms $\Ecal_{\lambda,m}$.
Background material is reviewed in the appendices: the coarea formula (Appendix~\ref{app:coarea}), subanalytic sets and stratifications (Appendix~\ref{app:suban}), and Dirichlet forms on stratified spaces (Appendix~\ref{app:dirichlet}).
We do assume familiarity with basic concepts from differential geometry as developed by e.g.\ \cite{MR2954043}.
All manifolds are second countable and Hausdorff, and submanifolds are understood to be embedded unless stated otherwise.

\subsection{Singular learning theory and the generalization puzzle} \label{sec:slt-connection}

Singular learning theory (SLT), due to \citet{MR2554932,MR3839272}, extends the classical large-sample theory of Bayesian inference to degenerate and nonidentifiable settings.
This includes many models relevant in machine learning, such as deep neural networks and other highly parameterized models.

A central result is Watanabe's \emph{free energy formula}, which relates the singularity types $(\lambda,m)$ associated with minimizers of the population loss to the large-sample behavior of the free energy (also known as the negative marginal log-likelihood or negative log-partition function).
The Bayesian posterior concentrates around the most singular stratum, and parameters from this stratum minimize the asymptotic Bayesian generalization error.\footnote{To be precise, the above is true under hypotheses including real-analyticity of the loss as well as realizability, meaning that the model nests the true data generating distribution.}
In a nutshell, \emph{singular models generalize well, and Bayesian inference locates the most singular model within a given parametric family.}

However, genuine Bayesian inference is not computationally feasible in practice.
Instead, one typically uses stochastic local search algorithms such as SGD (rather than Bayesian updating) to locate minimizers of an empirical loss based on a fixed data sample.
A priori there is little reason to expect such methods to prefer singular solutions.
Nonetheless, there is growing empirical evidence that they do, and that transitions to more singular strata at similar loss levels relate to increases in generalization ability; see e.g.\ \citet{chen2023dynamical,furman2024estimating,hoogland2025loss,wang2025differentiation}.

The present paper came about as an attempt to understand this phenomenon.
The closest prior work is that of \citet{li2022what}, who study SGD near zero loss using methods from stochastic analysis.
The parameter evolution under SGD with loss function $V \colon \R^d \to [0,\infty)$ is modeled by the continuous-time equation
\[
d\widetilde X = - \eta \nabla V(\widetilde X) dt + \eta \sigma(\widetilde X) d\widetilde B,
\]
where $\eta > 0$ is the learning rate and $(\sigma \sigma^\top)(x)$ is the noise covariance structure. Such models were first proposed by \citet{pmlr-v70-li17f,3122009.3208015}.
The simplest case is that of \emph{isotropic noise}, where $\sigma(x)$ is a multiple of the identity.
This case also models the SGLD algorithm \citep{3104482.3104568}.
To isolate the dynamics along the zero set of the loss, \citet{li2022what} study the small-$\eta$ limit of the time-changed process $X_t = \widetilde X_{t/\eta^2}$ which satisfies
\[
dX = - \frac{1}{\eta} \nabla V(X) dt + \sigma(X) dB
\]
with the new Brownian motion $B_t = \eta \widetilde B_{t/\eta^2}$.
In the case of suitably normalized isotropic noise, this is precisely \eqref{eq:original-SDE} with $\beta = 1/\eta$.
Assuming that $V^{-1}(0)$ is a smooth manifold along which the Hessian $\nabla^2 V$ is nondegenerate in the normal directions, \citet{li2022what} rely on work of \citet{MR1127717} to prove convergence to a limiting diffusion.
Among other things, they rigorously obtain the structural result that SGD is biased toward \emph{flat minima}, where the normal Hessian is small.
This behavior was previously observed by \citet{blanc2020implicit,xie2021diffusion,damian2021label}; see also \citet{hochreiter1997flat}.
There is also work showing convergence of SGD to singular solutions in the particular setting of quadratically parameterized models, without making a connection to SLT; see \citet{haochen2021shape,pesme2021implicit,pillaudvivien2022label}.

The present paper focuses on the SGLD (or isotropic noise SGD) case but allows for a general real-analytic loss, making a connection with the singularity types of SLT.
The upshot is that the SGLD state is pulled from less singular strata (regions of small local dimension) toward more singular strata (regions of large local dimension), and the strength of the pull explodes in the vicinity of the latter regions.
This may help explain why SGLD finds generalizing solutions.

It is worth highlighting some extensions, limitations, and further questions.

\begin{itemize}
\item SLT focuses on how \emph{large-sample} asymptotics are determined by the singularity structure of the \emph{population loss}.
In contrast, we study \emph{small learning rate} asymptotics and the influence on SGD (or SGLD) dynamics of the singularity structure of the \emph{fixed-sample empirical loss}.
It is an interesting problem to rigorously connect these perspectives.

\item Although this paper works with the set of global minimizers of $V$, by choosing the domain $D$ suitably the theory also applies to connected sets of local minimizers.
One wonders whether an extended notion of singularity structure could explain the dynamics of SGD along sets of near-minimizers as well.

\item Our methods do not immediately extend to more general noise models which yield non-symmetric diffusions.
A case of particular interest is $\sigma \sigma^\top = \nabla^2 V$.
Moreover, real analyticity of $V$ is essential for our results.
A separate theory would be required for models like neural networks with non-analytic activations.

\item In practice, SGD and SGLD are regularized using so-called \emph{weight decay}. To account for this, the continuous-time model \eqref{eq:original-SDE} should be replaced by
\[
dX = -\beta \nabla V(X) dt - \gamma X dt + \sqrt{2}\, dB
\]
for a fixed decay parameter $\gamma > 0$.
This diffusion is still symmetric, now with respect to the tilted weight $e^{-\beta V(x) - \gamma |x|^2/2}$.
The results of this paper then imply that Theorem~\ref{thm:main} still holds, with $\mu_{\lambda,m}$ replaced its tilted version $e^{-\gamma|x|^2/2} \mu_{\lambda,m}$.
In terms of the limiting stochastic dynamics, this corresponds to adding a drift term $-\gamma\, \tau(x) dt$, where $\tau(x)$ is the projection of $x$ onto the tangent space of the stratum containing $x$.

\item Our setup is reminiscent of, but not covered by, Freidlin--Wentzell metastability theory; see e.g.\ \citet{MR722136,MR2094397,MR2120991}.
More closely related, but still different, is their work on averaging principles for small perturbations of Hamiltonian systems \citep{MR1245308}.

\item As mentioned above, this paper does not prove weak convergence of $X^\beta$. Doing so would clearly be of interest. A natural strategy is to establish tightness of $\{X^\beta\}$ and strengthen the pointwise convergence of $\Ecal_\beta$ in Theorem~\ref{thm:main} to Mosco convergence in the sense of \citet{MR2015170}.
Furthermore, although we do identify candidate limit dynamics on each stratum $Z_{\lambda,m}$, rigorously concatenating these into a single process on $Z$ is left open.
\end{itemize}

\paragraph{Disclosure of AI use.}
Throughout the development of this paper, the author has made use of AI for proofreading and occasionally suggesting mathematical arguments, locating and summarizing references, and as an aid to learn material previously unfamiliar to the author, in particular related to algebraic geometry.

\section{Singularity types and the zero set} \label{sec:sing-types}

We now state the version of Hironaka's resolution of singularities that will be used in this paper.
We then review the construction of the LLC and multiplicities, and establish some basic topological and geometric properties.
While most of the material in this section is well-known, we do not know of a source where it is stated in the form required later.

We say that a nonnegative real-analytic function on an open set is \emph{nontrivial} if the zero set is nonempty and the function is not identically zero on any connected component of the set.

\begin{theorem}
Let $V \colon \R^d \to [0,\infty)$ be real-analytic and nontrivial in an open neighborhood of a compact set $K \subset \R^d$.
There exist a $d$-dimensional real-analytic manifold $M$, a relatively compact open neighborhood $W$ of $K$, a map $\pi \colon M \to W$, and a finite collection $\{E_\alpha\}$ of real-analytic hypersurfaces in $M$ with simple normal crossings carrying numerical data $(k_\alpha,h_\alpha) \in \N^* \times \N$, such that:
\begin{itemize}
\item \textbf{Resolution:} 
The zero set $W_0 = W\cap V^{-1}(0)$ of $V$ in $W$ lifts to
\[
\pi^{-1}(W_0) = \bigcup_\alpha E_\alpha
\]
and $\pi$ is real-analytic, proper (the inverse image of any compact set is compact), and its restriction to $M \setminus \pi^{-1}(W_0)$ is an isomorphism onto its image (invertible with real-analytic inverse). 

\item \textbf{Monomialization:} Each $p \in \pi^{-1}(W_0)$ belongs to some \emph{monomializing chart}, i.e., a chart $(U,y)$ with real-analytic coordinates $y = (y_1,\ldots,y_d)$ such that those $E_\alpha$ that meet $U$ are given by $\{y_i=0\}$, $i=1,\ldots,d' \le d$, and one has
\begin{align*}
V \circ \pi (y) = \prod_{i=1}^{d'} y_i^{2k_i} \qquad \text{and} \qquad
|D\pi(y)| = b(y) \prod_{i=1}^{d'} |y_i|^{h_i}
\end{align*}
where $k_i, h_i$ are the numerical data of $\{y_i=0\}$ and $b$ is a strictly positive real-analytic function.
Here $|D\pi|$ is the Jacobian determinant.
\end{itemize}
\end{theorem}

This formulation of the resolution theorem can be viewed as a ``global'' version of the ``local'' formulations of \citet{MR256156} and \citet{MR2554932}.
It can be deduced from \citet{BierstoneMilman1997}, specifically Theorem~1.10 and the subsequent sentence, followed by an argument that uses nonnegativity of $V \circ \pi$ to assert even exponents in its monomial form and remove a non-constant pre-factor. See \citet[Remark 2.6(8) and Remark 2.7]{MR2554932} for the latter argument.

\begin{remark}
In a monomializing chart with numerical data $k_i,h_i$ for $i \le d'$, we conventionally set $k_i = h_i = 0$ for $i = d'+1,\ldots,d$. With this convention we could equivalently let the products above range from $1$ to $d$.
\end{remark}

\begin{remark}
The union $\bigcup_\alpha E_\alpha$ is known as a \emph{divisor}, and the $E_\alpha$ are its \emph{components}.
The simple normal crossings (SNC) property of $\{E_\alpha\}$ means that at any point where two or more components intersect, they do so transversally in the sense that their conormals are linearly independent.
\end{remark}

\begin{remark} \label{rem:m-in-euclidean}
Thanks to the Grauert--Morrey embedding theorem \citep{MR98847}, we may and do assume for convenience that $M$ is embedded in an ambient Euclidean space.
This specifies a Hausdorff measure on $M$ and its submanifolds, which is convenient when applying the coarea formula.
\end{remark}

\begin{remark}
The SLT literature tends to work exclusively in local coordinate charts, avoiding the ``global'' formulation involving the components $E_\alpha$. For the present work, however, it is convenient to access the components $E_\alpha$ directly, not only via charts. This helps keep track of the components and their numerical data across charts, and simplifies the description of the strata of constant singularity type, the expressions for the constants in the Laplace asymptotics in Section~\ref{sec:laplace-asymptotics} below, as well as the analysis of stratifications in Section~\ref{sec:dirichlet-limit}.
\end{remark}

For any set $A$ of indices $\alpha$, we write
\[
E_A = \bigcap_{\alpha \in A} E_\alpha
\]
for the intersection of the components indexed by $A$. Because the $E_\alpha$ have simple normal crossings, $E_A$ is a closed real-analytic manifold of codimension $|A|$ whenever it is nonempty.

An important role is played by the numbers $(h_\alpha + 1)/(2k_\alpha)$. This is because they arise as local \emph{volume vanishing rates} for the sublevel sets $\{V \le \varepsilon\}$ as $\varepsilon \to 0$.
To see this in the simplest case, pick $\bar x \in Z$. Suppose that $\pi^{-1}(\bar x) = E_\alpha$ and that this fiber does not meet any other component $E_{\alpha'}$. For a ball $B$ centered at $\bar x$,
\[
\vol(B \cap \{V \le \varepsilon\}) = \int_{\pi^{-1}(B)} \bm1_{\{V \circ \pi \le \varepsilon\}}.
\]
If $B$ is small, any monomializing chart $(U,y)$ that meets $\pi^{-1}(\bar x)$ contributes to this integral by the amount
\[
\int_{U \cap \pi^{-1}(B)} \bm1_{\{y_1^{2 k_\alpha} \le \varepsilon\}} b(y) |y_1|^{h_\alpha} dy,
\]
which after a change of variables is seen to decay like $\varepsilon^{(h_\alpha+1)/(2k_\alpha)}$.
This explains the appearance of these numbers in the following definition, the terminology of which originates in singular learning theory \citep{MR2554932,MR3839272,lau2024locallearningcoefficientsingularityaware}.


\begin{definition}
For $x \in Z$, the \emph{local learning coefficient} at $x$ is defined as the leading volume vanishing rate among the components $E_\alpha$ that meet the fiber above $x$,
\[
\lambda(x) = \min_\alpha \left\{ \frac{h_\alpha + 1}{2k_\alpha} \colon x \in \pi(E_\alpha) \right\}.
\]
Call any $\alpha$ (or $E_\alpha$) that attains the minimum a \emph{critical} index (or component).
The \emph{multiplicity} at $x$ is then maximum number of critical components whose joint intersection meets the fiber above $x$,
\[
m(x) = \max_A \left\{|A| \colon x \in \pi(E_A) \text{ and } \frac{h_\alpha + 1}{2k_\alpha} = \lambda(x) \text{ for all } \alpha \in A \right\}.
\]
The pair $(\lambda(x), m(x))$ is called the \emph{singularity type} of $x$.
\end{definition}

Given any potential value $\lambda$ of the local learning coefficient, it is convenient to distinguish beween \emph{critical} (as above), \emph{sub-critical}, and \emph{super-critical} indices $\alpha$ (or components $E_\alpha$):
\begin{align*}
A_\crit &= \{\alpha \colon \frac{h_\alpha + 1}{2k_\alpha} = \lambda\}, \\
A_\subcr &= \{\alpha \colon \frac{h_\alpha + 1}{2k_\alpha} > \lambda\}, \\
A_\supcr &= \{\alpha \colon \frac{h_\alpha + 1}{2k_\alpha} < \lambda\}.
\end{align*}
We write $A_\crit(\lambda)$, $A_\subcr(\lambda)$, $A_\supcr(\lambda)$ whenever there is a need to emphasize the dependence on $\lambda$, though generally the value of $\lambda$ will be clear from the context.
With this notation the multiplicity can be written
\[
m(x) = \max \left\{|A| \colon A \subset A_\crit \text{ and } x \in \pi(E_A) \right\}.
\]
The sets $Z_{\lambda,m}$ in \eqref{eq:Z-lambda-m-def} can now be written
\[
Z_{\lambda,m} = \{x \in Z \colon \lambda(x) = \lambda \text{ and } m(x) = m\}.
\]

\begin{lemma} \label{lem:Z-lambda-m}
The singularity type is lower semicontinuous on $Z$ with respect to the ordering \eqref{eq:ordering}.
Moreover, $Z_{\lambda,m}$ is a locally closed subanalytic subset of $\R^d$ and admits the representation
\begin{equation} \label{eq:Z-lambda-m}
Z_{\lambda,m} = \pi(E_{\lambda,m}),
\end{equation}
where $E_{\lambda,m}$ is a codimension-$m$ real-analytic submanifold and subanalytic subset of $M$. It is given explicitly by
\begin{equation} \label{eq:E_A^*}
E_{\lambda,m} = \pi^{-1}(D) \cap \bigcup_{\substack{A \subset A_\crit \\ |A| = m}} E_A \setminus C_\textnormal{sat},
\end{equation}
where $C_\textnormal{sat} = \pi^{-1}(\pi(C))$ is the ``saturation'' along the fibers of $\pi$ of the set
\[
C = \Big( \bigcup_{\alpha \in A_\supcr} E_\alpha \Big) \cup \Big( \bigcup_{\substack{A \subset A_\crit \\ |A| > m}} E_A \Big).
\]
\end{lemma}

\begin{proof}[Proof of Lemma~\ref{lem:Z-lambda-m}]
For any fixed $\lambda$, the set $\{x \in Z \colon \lambda(x) \le \lambda\}$ is closed in $Z$. Indeed, it is the union of the finitely many sets $Z \cap \pi(E_\alpha)$ such that $(h_\alpha+1)/(2k_\alpha) \le \lambda$, each of which is closed in $Z$ since $\pi$ is proper.
Similarly, for fixed $(\lambda,m)$, the set $\{x \in Z \colon \lambda(x) = \lambda \text{ and } m(x) \ge m\}$ is closed in $Z_\lambda = \{x \in Z \colon \lambda(x) = \lambda\}$, being the union of the sets $Z_\lambda \cap \pi(E_A)$ such that $A \subset A_\crit(\lambda)$ and $|A| \ge m$.
From these two results we deduce that $\{x \in Z \colon (\lambda(x),m(x)) \le (\lambda,m)\}$ is closed in $Z$, as required for lower semicontinuity on $Z$.

Next, we show \eqref{eq:Z-lambda-m}. By definition, $\lambda(x) = \lambda$ and $m(x) = m$ means that the intersection $E_A$ of exactly $m$ critical $E_\alpha$ meets the fiber $\pi^{-1}(x)$, that no supercritical $E_\alpha$ meets the fiber, and that no intersection of more than $m$ critical $E_\alpha$ meets the fiber.
In symbols, $Z_{\lambda,m} = D \cap \bigcup_A \pi(E_A) \setminus \pi(C)$, where the union ranges over all $A \subset A_\crit$ with $|A|=m$.
Applying $\pi$ to both sides of \eqref{eq:E_A^*} one obtains precisely this set, thanks to the identity $\pi(B \setminus C_\textnormal{sat}) = \pi(B) \setminus \pi(C)$ for any set $B$.

It remains to show that $Z_{\lambda,m}$ is a locally closed subanalytic subset of $\R^d$, and that $E_{\lambda,m}$ is a codimension-$m$ real-analytic submanifold and subanalytic subset of $M$.
Since $C$ is a finite union of closed real-analytic submanifolds, it is a closed subanalytic set in view of \ref{app:def-semianalytic}, \ref{app:def-subanalytic}, and \ref{app:suban-properties-1} of Appendix~\ref{app:suban}.
Then so is $\pi(C)$ due to Appendix~\ref{app:suban}\ref{app:suban-images} using that $\pi$ is proper real-analytic, and the images $\pi(E_A)$ are too.
Further, $D$ is open and subanalytic by assumption.
The expression for $Z_{\lambda,m}$ in the previous paragraph thus shows that it is a Boolean combination of subanalytic sets, hence itself subanalytic, and it is locally closed thanks to lower semicontinuity of the singularity type.
Finally, $C_\textnormal{sat}$ is closed and $\pi^{-1}(D)$ is open by continuity of $\pi$, so we see from \eqref{eq:E_A^*} that $E_{\lambda,m}$ is the union of the mutually disjoint codimension-$m$ (unless empty) real-analytic submanifolds $\pi^{-1}(D) \cap (E_A \setminus C_\textnormal{sat})$ of $M$. 
Since these can be separated by open neighborhoods, the union $E_{\lambda,m}$ is itself such a manifold.
We also see that $E_{\lambda,m}$ is subanalytic, being a Boolean combination of subanalytic sets; here subanalyticity of $\pi^{-1}(D)$ and $C_\textnormal{sat}$ follows from Appendix~\ref{app:suban}\ref{app:suban-images}.
\end{proof}

\section{Laplace asymptotics} \label{sec:laplace-asymptotics}

We now turn to the large-$\beta$ asymptotics of Laplace integrals of the form
\[
\int_D f(x) e^{-\beta V(x)} dx
\]
for compactly supported test functions $f$.
The decay rate is governed by the most singular stratum $Z_{\lambda,m}$ that meets the support of $f$. 
At least, this is true generically; a delicate situation arises when $f$ vanishes everywhere on $Z_{\lambda,m}$ but $Z_{\lambda,m}$ still meets the boundary of $\{f > 0\}$.
To avoid such cases, we introduce a certain nested family of open subdomains $D_{\lambda,m} \subset D$.
Since the number of singularity types is finite by construction, we may enumerate them according to the ordering \eqref{eq:ordering} as
\[
(\lambda_1, m_1) < \cdots < (\lambda_N, m_N),
\]
ranging from most singular to least singular.
We then set
\[
D_{\lambda_i,m_i} = D \setminus \bigcup_{j < i} Z_{\lambda_j,m_j}, \quad i = 1, \ldots, N.
\]
Thus $D_{\lambda,m}$ is the complement in $D$ of the deeper strata of the zero set.
In particular, $D_{\lambda_1,m_1} = D$.
Lower semicontinuity of the singularity type implies that $Z_{\lambda,m}$ is closed in $D_{\lambda,m}$, and that $D_{\lambda,m}$ is open.
Our result on Laplace asymptotics is stated for functions with compact support in these subdomains.

\begin{theorem} \label{thm:laplace-asymptotics}
For each singularity type $(\lambda,m)$ there exists a Radon measure $\mu_{\lambda,m}$ on $Z_{\lambda,m}$ with full support such that for any $f \in C_c(D_{\lambda,m})$,
\begin{equation} \label{eq:j-th-asymptotic}
\beta^\lambda (\log \beta)^{1-m}  \int_D f(x) e^{-\beta V(x)} dx \to \int_{Z_{\lambda,m}} f \, d\mu_{\lambda,m}, \quad \beta \to \infty.
\end{equation}
Furthermore, the limit measure is the pushforward
\[
\mu_{\lambda,m} = \pi_* \rho_{\lambda,m}
\]
of a density $\rho_{\lambda,m}$ on the manifold $E_{\lambda,m}$ in \eqref{eq:E_A^*} that can be described as follows: for any $A \subset A_\crit$ with $|A|=m$, and any monomializing chart $(U,y)$ such that $E_A$ is $\{y_I = 0\}$ for some $I \subset \{1,\ldots,d'\}$, the density is given by
\begin{equation} \label{eq:rho-density}
\rho_{\lambda,m} = \frac{\Gamma(\lambda)}{(m-1)!} \, \frac{b(0_I, y_{-I})}{\prod_{i \in I} k_i} \prod_{i \notin I} |y_i|^{-2 k_i \lambda + h_i} dy_{-I}.
\end{equation}
\end{theorem}

\begin{remark}
The rate $\beta^\lambda (\log \beta)^{1-m}$ in \eqref{eq:j-th-asymptotic} is well-known and plays a crucial role in SLT.
For us, the main point of Theorem~\ref{thm:laplace-asymptotics} is that the limit can be represented by integration against a Radon measure $\mu_{\lambda,m}$ on $Z_{\lambda,m}$, and that we have detailed information about the form of this measure.
Further, by working stratum-by-stratum, we avoid having to derive full asymptotic expansions for general test functions $f \in C_c(D)$.
Such expansions, which are classical, contain a large number of terms that are incomparable to those appearing in \eqref{eq:j-th-asymptotic}, arising from certain arithmetic progressions of sub-leading poles in the zeta function $\zeta(s) = \int_D f(x) V(x)^s dx$.
Because of this, restricting to functions with support in $D_{\lambda,m}$ considerably simplifies the resulting representation.
General asymptotic expansions for Laplace integrals, including formulas closely related to the one obtained in Lemma~\ref{lem:zeta-alpha-residue} below, are classical and discussed by \citet[Lemma~7.4 and Theorem~7.4]{MR2919697} and \citet[Ch.~3]{lin2011algebraic}.
Lastly, a weak convergence approach to Laplace asymptotics was developed by \cite{MR602391} for potentials with more structured zero sets; see also \cite{MR2658170,MR4474549}.
\end{remark}

\begin{remark}
An equivalent formulation of the convergence statement in Theorem~\ref{thm:laplace-asymptotics} is that the rescaled measures $c_\beta e^{-\beta V(x)} dx$ with $c_\beta = \beta^\lambda (\log \beta)^{1-m}$ converge vaguely on $D_{\lambda,m}$ to $\mu_{\lambda,m}$.
\end{remark}

\begin{remark} \label{rem:blowup}
It is instructive to examine the blowup rates of $\rho_{\lambda,m}$ in \eqref{eq:rho-density}.
Consider some $A \subset A_\crit$ with $|A| = m$ and a monomializing chart $(U,y)$ where $E_A$ is given by $\{y_I = 0\}$ for some index set $I$.
Suppose another component $E_\alpha$ also meets $U$ and is given by $\{y_j = 0\}$ there for some $j \notin I$.
Denote its volume vanishing rate by
\[
\lambda_\alpha = \frac{h_\alpha + 1}{2 k_\alpha} = \frac{h_j + 1}{2 k_j}.
\]
The exponent of $|y_j|$ in \eqref{eq:rho-density} can then be written
\[
 - p_\alpha = - 1 - \left(\frac{\lambda}{\lambda_\alpha} - 1\right)(h_\alpha + 1).
\]
This exponent is less than, equal to, or greater than $-1$ exactly according to whether $E_\alpha$ is supercritical, critical, or subcritical.
\begin{itemize}
\item In the subcritical case $\lambda_\alpha > \lambda$, the set $\{y_I = 0\} \cap \{y_j = 0\}$, which is the part of $E_A \cap E_\alpha$ that meets $U$, is still mapped by $\pi$ to the stratum $Z_{\lambda,m}$ (unless more singular contributions arise from other components, possibly not meeting $U$).
The factor $|y_j|^{-p_\alpha}$ is integrable in this case.

\item In the critical case $\lambda_\alpha = \lambda$, points in $\{y_I = 0\} \cap \{y_j = 0\}$ contribute to increasing the multiplicity to $m+1$ while leaving the LLC unchanged. This shows up as a logarithmic divergence $|y_j|^{-1}$ in \eqref{eq:rho-density}. These points thus map to $Z_{\lambda,m+1}$, or to a deeper stratum in the presence of other, more singular components.

\item In the supercritical case $\lambda_\alpha < \lambda$, $E_\alpha$ contributes a strictly smaller LLC, so points in $\{y_I = 0\} \cap \{y_j = 0\}$ map to $Z_{\lambda_\alpha,1}$ or to a deeper stratum. We now have a polynomial divergence $|y_j|^{-p_\alpha}$ whose severity depends on the relative size of $\lambda$ and $\lambda_\alpha$, as well as the Jacobian exponent $h_\alpha$.
\end{itemize}

By lower semicontinuity of the singularity type (see Lemma~\ref{lem:Z-lambda-m}), any point $x_0$ of the frontier $\overline{Z_{\lambda,m}} \setminus Z_{\lambda,m}$ that also belongs to $D$ must be inside some deeper stratum $Z_{\lambda',m'}$.
Since $x_0$ belongs to $\pi(\overline E_{\lambda,m})$, it is natural to ask whether the above blowup behavior is realized along sequences in $E_{\lambda,m}$.
In other words, does $\pi^{-1}(x_0) \cap \overline E_{\lambda,m}$ meet the set $C$ in Lemma~\ref{lem:Z-lambda-m}?
Relatedly, it is natural to expect, but not clear how to prove, that $\mu_{\lambda,m}$ assigns infinite mass to any neighborhood of $x_0$.
We do not pursue these questions further.
\end{remark}

The proof of Theorem~\ref{thm:laplace-asymptotics} follows a standard approach to determining the asymptotic behavior of Laplace integrals; see \citet{MR2919697} for the classical theory, and \citet{MR2554932} and \citet{lin2011algebraic} for its use in SLT.
Using the Mellin--Barnes formula for the exponential function,
\begin{equation} \label{eq:mellin-barnes}
e^{-u} = \frac{1}{2 \pi \i} \int_{c-\i \infty}^{c+\i \infty} \Gamma(-s) u^s ds, \quad c < 0, \quad u > 0,    
\end{equation}
one writes the Laplace integral as
\begin{equation} \label{eq:laplace-in-terms-of-zeta}
\int_D f(x) e^{-\beta V(x)} dx = \frac{1}{2\pi\i} \int_{c-\i \infty}^{c+\i \infty} \Gamma(-s) \beta^s \zeta(s) ds
\end{equation}
in terms of the zeta function
\begin{equation} \label{eq:zeta-global}
\zeta(s) = \int_D f(x) V(x)^s dx.
\end{equation}
The resolution theorem is then used to decompose $\zeta(s)$ into a sum of contributions, each of which can be analyzed in resolved coordinates.
Concretely, let $M$ and $\pi$ be as in the resolution theorem.
Fix a finite collection of monomializing charts $(U,y)$ that cover $\pi^{-1}(Z)$, and let $\{\chi_U\}$ be a smooth partition of unity subordinate to this cover.
Suppose also that these charts cover $\pi^{-1}(\supp(f))$; this turns out to be without loss of generality.
One then has
\[
\zeta(s) = \int_M (f V^s) \circ \pi = \sum_U \int_U \chi_U (f V^s) \circ \pi = \sum_U \zeta_U(s),
\]
where each ``localized'' zeta function $\zeta_U$ can be expanded in the chart $(U,y)$ as
\begin{equation} \label{eq:zeta-localized}
\zeta_U(s) = \int_{\R^d} \chi_U(y) (f \, V^s) \circ \pi(y) |D\pi(y)| dy =  \int_{\R^d} f_U(y) \prod_{i=1}^d |y_i|^{2k_i s + h_i} dy
\end{equation}
for the $C^\infty_c(U)$ function $f_U(y) = \chi_U(y) f \circ \pi(y) b(y)$.
Here $b(y)$ is the strictly positive real-analytic function appearing in the resolution theorem.
Combining \eqref{eq:laplace-in-terms-of-zeta} and \eqref{eq:zeta-localized} one obtains the formula
\begin{equation}
\int_D f(x) e^{-\beta V(x)} dx = \frac{1}{2\pi\i} \sum_U \int_{c-\i \infty}^{c+\i \infty} \Gamma(-s) \beta^s \zeta_U(s) ds.
\end{equation}
The asymptotic behavior can now be determined by computing the residue of each product $\Gamma(-s) \beta^s \zeta_U(s)$ at its leading pole, along with the location and multiplicity of that pole.

At this point the analysis becomes somewhat more involved because, first, we are interested in the vanishing rate of the Laplace integral \emph{for $f$ supported in $D_{\lambda,m}$}, and, second, we require fairly detailed information about the form of the limit after rescaling.
Lemma~\ref{lem:zeta-alpha-residue} below supplies the required information.

First, we record a consequence of the support condition on $f$.
Since we will work in the coordinates of a particular chart $(U,y)$, it is convenient to map those indices in $A_\crit$, $A_\subcr$, and $A_\supcr$ whose components meet $U$ to the corresponding coordinate indices in the chart.
Specifically, we let
\[
I_\crit, \quad I_\subcr, \quad I_\supcr
\]
consist of those indices $i$ such that, in the chart $(U,y)$, the critical (sub-critical, super-critical) components $E_\alpha$ are given by $\{y_i = 0\}$.

\begin{lemma} \label{lem:supp-epsilon}
Fix a singularity type $(\lambda,m)$, a function $f \in C^\infty_c(D_{\lambda,m})$, and a monomializing chart $(U,y)$. Then for any compact set $K \subset U$ there is $\varepsilon > 0$ such that
\begin{equation} \label{eq:supp-disjoint-I-sup}
K \cap \supp(f \circ \pi) \subset \{|y_i| \ge \varepsilon \text{ for all } i \in I_\supcr\}
\end{equation}
and
\begin{equation} \label{eq:supp-psi-alpha-on-large-I}
K \cap \supp(f \circ \pi) \cap \{y_I = 0\} \subset \{|y_i| \ge \varepsilon \text{ for all } i \in I_\crit \setminus I\}, \quad I \subset I_\crit, \ |I| = m.
\end{equation}
\end{lemma}

\begin{proof}
Because $\supp(f)$ does not meet deeper strata, and because $p \in \supp(f \circ \pi)$ implies $\pi(p) \in \supp(f)$,
it follows from the definition of the LLC and multiplicity that $\supp(f \circ \pi)$ cannot meet any supercritical components, nor more than $m$ intersecting critical components. More precisely,
\[
\supp(f \circ \pi) \cap E_\alpha = \emptyset, \quad \alpha \in A_\supcr
\]
and
\[
\supp(f \circ \pi) \cap E_A \cap E_\alpha = \emptyset, \quad A \subset A_\crit, \ |A| = m, \ \alpha \in A_\crit \setminus A.
\]
We may write this equivalently as
\[
\supp(f \circ \pi) \subset \bigcap_{\alpha \in A_\supcr} E_\alpha^c
\]
and
\[
\supp(f \circ \pi) \cap E_A \subset \bigcap_{\alpha \in A_\crit \setminus A} E_\alpha^c, \quad A \subset A_\crit, \ |A| = m.
\]
In the coordinates of $(U,y)$, this yields \eqref{eq:supp-disjoint-I-sup} and \eqref{eq:supp-psi-alpha-on-large-I} with ``$\ge \varepsilon$'' replaced by ``$> 0$''. The existence of a uniform lower bound $\varepsilon > 0$ follows from compactness of $K$.
\end{proof}

The following lemma is key to the refined asymptotics of Theorem~\ref{thm:laplace-asymptotics}.

\begin{lemma} \label{lem:zeta-alpha-residue}
Fix a singularity type $(\lambda,m)$, a function $f \in C^\infty_c(D_{\lambda,m})$, and a monomializing chart $(U,y)$.
The zeta function $\zeta_U$ in \eqref{eq:zeta-localized} is holomorphic in $\real(s) > -\lambda$ and admits a Laurent expansion around $-\lambda$ of the form
\[
\zeta_U(s) = \frac{c_U}{(s+\lambda)^m} + O\left( \frac{1}{(s+\lambda)^{m-1}} \right),
\]
where the dominant Laurent coefficient is given by
\[
c_U = \sum_{\substack{I \subset I_\crit \\ |I| = m}} \frac{1}{\prod_{i \in I} k_i} \int_{\R^{d-m}} f_U(0_I, y_{-I}) \prod_{i \notin I} |y_i|^{- 2k_i \lambda + h_i} \, dy_{-I}.
\]
If $|I_\crit| < m$, the sum is empty and we have $c_U = 0$.
\end{lemma}

\begin{proof}[Proof of Lemma~\ref{lem:zeta-alpha-residue}]
\textit{Step 1.}
First observe that no singularities can arise from super-critical coordinates due to the support condition on $f$.
Indeed, applying Lemma~\ref{lem:supp-epsilon} with $K = \supp(\chi_U)$, we get $\varepsilon > 0$ such that, thanks to \eqref{eq:supp-disjoint-I-sup},
\begin{align*}
\zeta_U(s)
&= \int_{\R^d} f_U(y) \prod_{i \notin I_\supcr} |y_i|^{2k_i s + h_i}  \prod_{i \in I_\supcr} (\varepsilon \vee |y_i|)^{2k_i s + h_i} dy.
\end{align*}

\textit{Step 2.}
We will study the Laurent expansion of $\zeta_U$ by viewing $\zeta_U(s) = \langle T_s, f_U \rangle$ as the action on $f_U$ of a tensor product $T_s = \bigotimes_{i=1}^d T_{i,s}$ of certain generalized functions (or distributions) that we now define.
We point the reader to \citet[Ch.~I, \S\S 3.1--3.3 and Appendix~2]{MR3469458} for the required background in the case $2k_i = 1$, $h_i=0$; the general case is derived in the same way.
Set
\[
T_{s,i}(y_i) =
\begin{cases}
|y_i|^{2k_i s + h_i}, & i \notin I_\supcr, \\[1ex]
(\varepsilon \vee |y_i|)^{2k_i s + h_i}, & i \in I_\supcr,
\end{cases}
\]
regarded as complex-valued generalized functions acting on $C^\infty_c(\R)$, depending meromorphically on $s \in \C$.
For any $i \in I_\supcr$ and any ``passive'' $i$ (i.e.\ one that does not correspond to any component $E_\alpha$, so that $k_i = h_i = 0$), $T_{i,s}$ is actually entire in $s$.
In particular, expanding around $s = -\lambda$ gives
\[
T_{s,i}(y_i) =
\begin{cases}
1, & i \text{ passive}, \\[1ex]
(\varepsilon \vee |y_i|)^{-2k_i \lambda + h_i} + O(s+\lambda), & i \in I_\supcr.
\end{cases}
\]
Next, for $i \in I_\crit \cup I_\subcr$, $T_{i,s}$ is meromorphic and its poles (all of them simple) constitute the arithmetic progression of those $s$ for which $2k_i s + h_i$ is a negative odd integer. In particular, the leading pole is $-(h_i+1)/(2k_i)$, which is equal to $-\lambda$ for $i \in I_\crit$ and strictly less than $-\lambda$ for $i \in I_\subcr$.
Thus for $i \in I_\subcr$, $T_{i,s}$ is holomorphic in the halfplane $\real(s) > - \lambda - \eta$ for some $\eta > 0$, and one has
\[
T_{s,i}(y_i) = |y_i|^{-2k_i\lambda + h_i} + O(s + \lambda), \quad i \in I_\subcr.
\]
For $i \in I_\crit$ one has the distributional Laurent expansion
\[
T_{s,i}(y_i) = |y_i|^{2k_i(s+\lambda) -  1} = \frac{\delta(y_i)}{k_i} (s + \lambda)^{-1} + \text{Pf$_\varepsilon$\,} |y_i|^{-1} + 2 \log \varepsilon \, \delta(y_i) + O(s + \lambda),
\]
where $\text{Pf$_\varepsilon$\,} |y_i|^{-1}$ is the finite part of $|y_i|^{-1}$ with cutoff $\varepsilon$, acting on test functions $\varphi$ by
\[
\langle \text{Pf$_\varepsilon$\,} |y_i|^{-1}, \varphi \rangle = \int_{|y_i| \le \varepsilon} \frac{\varphi(y_i) - \varphi(0)}{|y_i|} dy_i + \int_{|y_i| > \varepsilon} \frac{\varphi(y_i)}{|y_i|} dy_i.
\]
\citet[Ch.~I, \S 4.3, Eqs.~(11)--(12)]{MR3469458} derive this formula for $2k_i = 1$, $h_i = 0$, $\varepsilon = 1$, and the general case is derived in the same manner.
Note that the above expansion holds for any $\varepsilon > 0$, but we specifically use the $\varepsilon$ from Step~1.
Before moving on, note also that the above shows each $T_{s,i}$ to be holomorphic in $\real(s) > -\lambda$. Thus so is $T_s$ and hence $\zeta_U$, as claimed in the statement of the lemma.

\textit{Step 3.}
We now consider the principal part of the distributional Laurent expansion of $T_s$ around $-\lambda$.
The expansion is obtained by plugging in the individual expansions of the $T_{i,s}$ and expanding using linearity of the tensor product.
In doing so, contributions to the principal part can only arise from the principal term of $T_{i,s}$ for $i \in I_\crit$.
Thus for each nonempty $I \subset I_\crit$, one obtains terms of the form
\begin{align*}
\frac{C(y)}{(s + \lambda)^l}, \quad C(y) &= \bigg( \bigotimes_{i \in I} \frac{\delta(y_i)}{k_i} \bigg)
\otimes \bigg( \bigotimes_{i \notin I} S_i(y_i) \bigg),
\end{align*}
for some $l \le |I|$ and some distributions $S_i$ that arise from constant or $O(s+\lambda)$ terms in the expansions of the $T_{i,s}$ for $i \notin I$. 
Note that strict inequality $l < |I|$ can arise due to products of $(s+\lambda)^{-1}$ terms (from $i \in I$) and $O(s+\lambda)$ terms (from $i \notin I$).

Due to \eqref{eq:supp-psi-alpha-on-large-I} in Lemma~\ref{lem:supp-epsilon}, if $|I| > m$, then $\langle C,f_U\rangle = 0$.
Thus, the most singular term in the Laurent expansion of $\zeta_U$ will arise from choosing $I \subset I_\crit$ with $|I| = m$, and combining the principal terms of the $T_{i,s}$ for $i \in I$ with the constant terms of the $T_{i,s}$ for $i \notin I$.
This part of the expansion of $T_s$ is given by
\[
\frac{1}{(s + \lambda)^m} \sum_{\substack{I \subset I_\crit \\ |I| = m}} C_I(y)
\]
where
\begin{align*}
C_I(y) = \bigotimes_{i \in I} \frac{\delta(y_i)}{k_i} & \bigotimes_{i \in I_\crit \setminus I}\left(\text{Pf$_\varepsilon$\,} |y_i|^{-1} + 2 \log \varepsilon \, \delta(y_i)\right) \\
&\ \, \bigotimes_{i \in I_\supcr} (\varepsilon \vee |y_i|)^{-2k_i\lambda + h_i} \\
&\! \! \bigotimes_{i \notin I_\crit \cup I_\supcr} |y_i|^{-2k_i\lambda + h_i}.
\end{align*}
Due to the properties \eqref{eq:supp-disjoint-I-sup} and \eqref{eq:supp-psi-alpha-on-large-I} of the support of $f_U$ (still taking $K = \supp(\chi_U)$), one obtains
\begin{align}
\langle C_I, f_U \rangle &= \frac{1}{\prod_{i \in I} k_i} \int_{\R^{d-m}} f_U(0_I, y_{-I}) \prod_{i \in I_\crit \setminus I} |y_i|^{-1} \\
& \qquad \qquad \qquad \qquad \qquad \ \ \times \ \ \,  \prod_{i \in I_\supcr} |y_i|^{-2k_i \lambda + h_i} \nonumber \\
& \qquad \qquad \qquad \qquad \qquad \ \ \times \prod_{i \notin I_\crit \cup I_\supcr} |y_i|^{-2k_i \lambda + h_i} dy_{-I} \nonumber \\
&= \frac{1}{\prod_{i \in I} k_i}\int_{\R^{d-m}} f_U(0_I, y_{-I}) \prod_{i \notin I} |y_i|^{-2k_i \lambda + h_i} dy_{-I}. \label{eq:c-I-on-psi-alpha}
\end{align}
We conclude that the Laurent expansion of $\zeta_U$ around $-\lambda$ is of the form
\[
\zeta_U(s) = \frac{1}{(s+\lambda)^m} \sum_{\substack{I \subset I_\crit \\ |I| = m}} \langle C_I, f_U \rangle + O\left( \frac{1}{(s+\lambda)^{m-1}}\right).
\]
After substituting \eqref{eq:c-I-on-psi-alpha}, this is precisely the claimed expression.
\end{proof}

We are now ready to prove Theorem~\ref{thm:laplace-asymptotics}. 
As discussed above, the proof follows a classical approach.

\begin{proof}[Proof of Theorem~\ref{thm:laplace-asymptotics}]
Fix a singularity type $(\lambda,m)$ and a function $f \in C^\infty_c(D_{\lambda,m})$.
We also fix $M, \pi$ as in the resolution theorem and a finite collection of monomializing charts $(U,y)$ that cover $\pi^{-1}(Z)$, along with a smooth partition of unity $\{\chi_U\}$ subordinate to this cover.
We may and do assume that $\pi^{-1}(\supp(f))$ is covered by these charts, because contributions to the integral coming from parts of the support bounded away from $Z$ will vanish at an exponential rate as $\beta \to \infty$.

The proof rests on the following calculation, whose steps are justified below. Let $c \in (-\lambda, 0)$ and recall the zeta functions $\zeta$ and $\zeta_U$ in \eqref{eq:zeta-global} and \eqref{eq:zeta-localized}. We have
\begin{align}
\int_D f(x) e^{-\beta V(x)} dx
&= \frac{1}{2 \pi \i} \int_{c - \i \infty}^{c + \i\infty} \Gamma(-s) \beta^s \zeta(s) ds \label{eq:asymp-calculation-step-1} \\
&= \frac{1}{2 \pi \i} \sum_U \int_{c - \i \infty}^{c + \i\infty} \Gamma(-s) \beta^s \zeta_U(s) ds \label{eq:asymp-calculation-step-2} \\
&= \sum_U \Res_{-\lambda}(\Gamma(-s) \beta^s \zeta_U(s)) + R(\beta) \label{eq:asymp-calculation-step-3} \\
&= \frac{\Gamma(\lambda)}{(m-1)!} \beta^{-\lambda} (\log \beta)^{m-1} \sum_U c_U + R(\beta) \label{eq:asymp-calculation-step-4} 
\end{align}
where $R(\beta) = o(\beta^{-\lambda}(\log \beta)^{m-1})$ and $c_U$ is given in Lemma~\ref{lem:zeta-alpha-residue}. 
We now justify each step in turn.

\eqref{eq:asymp-calculation-step-1}: 
It follows from \eqref{eq:zeta-localized}, which initially holds for $\real(s) > 0$, that $\zeta_U$ has a meromorphic continuation to all of $\C$ with poles at certain negative rationals.
Thanks to Lemma~\ref{lem:zeta-alpha-residue} the largest pole is at most $-\lambda$, so the formula \eqref{eq:zeta-localized} actually holds, at least, for $\real(s) > -\lambda$.
Furthermore, standard arguments show that $\zeta_U$ is bounded on vertical strips away from the poles in the sense that for any real $a<b$ and $\delta > 0$ there is a constant $C_U = C_U(a,b,\delta)$ such that
\begin{equation}\label{eq:zeta_alpha_bdd}
|\zeta_U(s)| \le C_U \text{ if $a < \real(s) < b$ and $|s-p| > \delta$ for every pole $p$.}
\end{equation}
The original zeta function $\zeta = \sum_U \zeta_U$ then also has a meromorphic continuation to all of $\C$, its largest pole is at most $-\lambda$, the formula \eqref{eq:zeta-global} holds for $\real(s) > -\lambda$, and the boundedness property \eqref{eq:zeta_alpha_bdd} holds with $\zeta$ and $C = \sum_U C_U$ in place of $\zeta_U$ and $C_U$.

With these facts in mind, and recalling that $c \in (-\lambda,0)$, we have from the Mellin--Barnes formula \eqref{eq:mellin-barnes} with $u = \beta V(x)$ for $x$ outside the zero set $Z$ (which is a Lebesgue nullset), Fubini's theorem, and \eqref{eq:zeta-global} that
\begin{align*}
\int_D f(x) e^{-\beta V(x)} dx
&= \frac{1}{2 \pi \i} \int_D f(x) \int_{c - \i \infty}^{c + \i\infty} \Gamma(-s) \beta^s V(x)^s ds \, dx \\
&= \frac{1}{2 \pi \i} \int_{c - \i \infty}^{c + \i\infty} \Gamma(-s) \beta^s \int_D f(x) V(x)^s dx \, ds \\
&= \frac{1}{2 \pi \i} \int_{c - \i \infty}^{c + \i\infty} \Gamma(-s) \beta^s \zeta(s) ds.
\end{align*}
The use of Fubini's theorem is justified by observing that
\[
\int_D \int_{c-\i \infty}^{c+\i \infty} |f(x)| |\Gamma(-s)| (\beta V(x))^c ds \, dx \le \beta^c \int_D |f(x)| V(x)^c dx \int_{c-\i \infty}^{c+\i \infty} |\Gamma(-s)| ds.
\]
The first integral on the right-hand side is bounded by $\zeta(c;g)$ (defined by taking $g$ in place of $f$ in \eqref{eq:zeta-global}), where $g \in C^\infty_c(D)$ is nonnegative and equal to $\|f\|_\infty$ on the support on $f$.
The second integral is finite since the Gamma function decays exponentially along vertical lines and does not have poles in the right half-plane.
This proves \eqref{eq:asymp-calculation-step-1}.

\eqref{eq:asymp-calculation-step-2} and \eqref{eq:asymp-calculation-step-3}: The first of these, \eqref{eq:asymp-calculation-step-2}, is immediate from $\zeta = \sum_U \zeta_U$. To obtain \eqref{eq:asymp-calculation-step-3} we fix $U$ and shift the contour of integration to the left past the pole at $s=-\lambda$. This yields the identity
\[
\frac{1}{2 \pi \i} \int_{c - \i \infty}^{c + \i\infty} \Gamma(-s) \beta^s \zeta_U(s) ds
= \Res_{-\lambda}(\Gamma(-s) \beta^s \zeta_U(s)) + R_U(\beta),
\]
where
\[
R_U(\beta) = \frac{1}{2 \pi \i} \int_{c' - \i \infty}^{c' + \i\infty} \Gamma(-s) \beta^s \zeta_U(s) ds
\]
for any $c' \in (-\lambda', -\lambda)$, where $-\lambda'$  is the next pole to the left of $-\lambda$. The boundedness property \eqref{eq:zeta_alpha_bdd} and the exponential decay of the Gamma function along vertical lines ensures that this is rigorous, and that one has the bound
\[
|R_U(\beta)| \le C'_U \beta^{-c'}, \quad C'_U = \frac{1}{2\pi} \int_{c' - \i \infty}^{c' + \i\infty} |\Gamma(-s)| |\zeta_U(s)| ds < \infty.
\]
This shows that $R_U(\beta) = o(\beta^{-\lambda}(\log \beta)^{m-1})$ and completes the proof of \eqref{eq:asymp-calculation-step-3}.

\eqref{eq:asymp-calculation-step-4}: We must compute the asymptotics of the residue in \eqref{eq:asymp-calculation-step-3}, i.e., the coefficient in front of $(s + \lambda)^{-1}$ in the Laurent expansion of $\Gamma(-s) \beta^s \zeta_U(s)$ around $-\lambda$. To do so, first note that $\Gamma(-s) \beta^s$ is holomorphic in a neighborhood of $-\lambda$ with power series expansion
\[
\Gamma(-s) \beta^s = \sum_{k=0}^\infty a_k (s + \lambda)^k
\]
whose coefficients are given by
\[
a_k = \frac{1}{k!} \frac{d^k}{ds^k}\Big( \Gamma(-s) \beta^s) \Big)_{s = -\lambda} = \frac{\Gamma(\lambda)}{k!} (\log \beta)^k \beta^{-\lambda} + o( (\log \beta)^k \beta^{-\lambda} ).
\]
Thanks to Lemma~\ref{lem:zeta-alpha-residue}, $\zeta_U$ is holomorphic in a punctured neighborhood of $-\lambda$, with an order-$m$ pole at $-\lambda$. Let $b_k$, $k \ge -m$, be the coefficients of its Laurent expansion around $-\lambda$.
The coefficient of $(s+\lambda)^{-1}$ in the Laurent expansion of the product $\Gamma(-s) \beta^s \zeta_U(s)$ is then
\[
\sum_{k=0}^{m-1} a_k b_{-1-k} = \frac{\Gamma(\lambda)}{(m-1)!} (\log \beta)^{m-1} \beta^{-\lambda}b_{-m} + o( (\log \beta)^{m-1} \beta^{-\lambda} ).
\]
This completes the proof of \eqref{eq:asymp-calculation-step-4} because, again by Lemma~\ref{lem:zeta-alpha-residue}, $b_{-m} = c_U$.

At this point we have shown that
\[
\frac{(m-1)!}{\Gamma(\lambda)} \beta^\lambda (\log \beta)^{1-m} \int_{\R^d} f(x) e^{-\beta V(x)} dx \to \sum_U c_U, \quad \beta \to \infty,
\]
with $c_U$ given by Lemma~\ref{lem:zeta-alpha-residue}.
This is however precisely \eqref{eq:j-th-asymptotic} for $f \in C^\infty_c(D_{\lambda,m})$, because the formula for $c_U$ in Lemma~\ref{lem:zeta-alpha-residue} shows that
\[
\frac{\Gamma(\lambda)}{(m-1)!} \sum_U c_U = \int_{Z_{\lambda,m}} f d\mu_{\lambda,m}
\]
for the measure $\mu_{\lambda,m}$ in the statement of the theorem.
An approximation argument extends the convergence to all $f \in C_c(D_{\lambda,m})$.

Lastly, to see that $\mu_{\lambda,m}$ has full support, pick any set $O$ that is nonempty and open in $Z_{\lambda,m}$.
By continuity, $\pi^{-1}(O)$ is open in $E_{\lambda,m}$.
Since the density $\rho_{\lambda,m}$ is strictly positive except perhaps on a finite union of lower-dimensional submanifolds of $E_{\lambda,m}$, its integral over any open set is strictly positive.
We deduce that $\mu_{\lambda,m}(O) > 0$.
This completes the proof of the theorem.
\end{proof}

\section{The limiting Dirichlet forms} \label{sec:dirichlet-limit}

Let $f,g$ be two $C^1$ functions with compact support in $D_{\lambda,m}$.
The Laplace asymptotics in Theorem~\ref{thm:laplace-asymptotics} implies that after rescaling, the bilinear expression
\[
\Ecal_\beta(f,g) = \int_D \nabla f \cdot \nabla g \, e^{-\beta V} dx
\]
converges as $\beta \to \infty$ to
\begin{equation} \label{eq:limiting-dirichlet-form}
\Ecal_{\lambda,m}(f,g) = \int_{Z_{\lambda,m}} \nabla f \cdot \nabla g \, d\mu_{\lambda,m},
\end{equation}
where $\mu_{\lambda,m}$ is the measure on $Z_{\lambda,m}$ given in Theorem~\ref{thm:laplace-asymptotics}.
Recall that this measure has full support.

We wish to view $\Ecal_{\lambda,m}$ as a bilinear form on a suitable domain in $L^2(Z_{\lambda,m}, \mu_{\lambda,m})$. 
However, because $Z_{\lambda,m}$ is a lower-dimensional set, there will be functions $f$ and $g$ that vanish on $Z_{\lambda,m}$ but whose gradients do not.
For this reason we need to impose further restrictions on $f$ and $g$ to make \eqref{eq:limiting-dirichlet-form} well-defined on $L^2(Z_{\lambda,m}, \mu_{\lambda,m})$.

We address this issue by fixing a Whitney stratification $\Scal_{\lambda,m}$ of $Z_{\lambda,m}$ and requiring $f$ and $g$ to have \emph{stratified gradients}. 
This means that for every $x \in Z_{\lambda,m}$, $\nabla f(x)$ belongs to the tangent space at $x$ of the stratum $S$ that contains $x$; see Appendix~\ref{app:suban}.
If such a function vanishes on $Z_{\lambda,m}$, it is clear that its gradient does too.
We thus take the domain of $\Ecal_{\lambda,m}$ to be the restrictions to $Z_{\lambda,m}$ of $C^1_c$ functions with stratified gradients,
\[
\Dcal_{\lambda,m} = C^1_{c,\,\textnormal{strat}}(\Scal_{\lambda,m}) = \{f|_{Z_{\lambda,m}} \colon f \in C^1_c(D_{\lambda,m}) \text{ and } f \text{ has stratified gradient}\}.
\]
With this domain, $\Ecal_{\lambda,m}$ is a well-defined symmetric bilinear form on $L^2(Z_{\lambda,m}, \mu_{\lambda,m})$, and its domain is dense thanks to Lemma~\ref{lem:c-strat}\ref{lem:c-strat-4}.

So far, the specific choice of Whitney stratification did not play a role, and we know that some such stratification must exist because $Z_{\lambda,m}$ is subanalytic.
The following result shows that with an appropriate choice of stratification, more can be said.

\begin{theorem} \label{thm:limiting-dirichlet-form}
The set $Z_{\lambda,m}$ admits a subanalytic Whitney stratification $\Scal_{\lambda,m}$ satisfying the frontier condition such that the form $(\Ecal_{\lambda,m}, \Dcal_{\lambda,m})$ is closable and the closure is a strongly local regular Dirichlet form.
\end{theorem}

\begin{proof}
This follows from Lemma~\ref{lem:whitney-strat-for-z-m-lambda} and Lemma~\ref{lem:closability-stratified}.
\end{proof}

The following lemma drives the proof of Theorem~\ref{thm:limiting-dirichlet-form}. It implies that $Z_{\lambda,m}$ can be stratified in a way that makes $\Ecal_{\lambda,m}$ closable.

\begin{lemma} \label{lem:whitney-strat-for-z-m-lambda}
$Z_{\lambda,m}$ admits a subanalytic Whitney stratification such that the restriction of $\mu_{\lambda,m}$ to any stratum $S$ is either zero or satisfies the Hamza condition~\eqref{eq:hamza-condition}. Further, the stratification can be chosen to have connected strata and satisfy the frontier condition.
\end{lemma}

\begin{proof}
The starting point is Theorem~\ref{thm:laplace-asymptotics}, which expresses $\mu_{\lambda,m}$ as the pushforward under $\pi$ of a certain density $\rho_{\lambda,m}$ on the codimension-$m$ real-analytic manifold $E_{\lambda,m}$ occurring in the representation of $Z_{\lambda,m}$ in Lemma~\ref{lem:Z-lambda-m}; see \eqref{eq:Z-lambda-m}--\eqref{eq:E_A^*}.

The manifold $E_{\lambda,m}$ may have several connected components, and while these components all have dimension $d-m$, the map $\pi$ may send different components to sets of different dimension. For this reason it is convenient to decompose $E_{\lambda,m}$ into its connected components, and refine \eqref{eq:Z-lambda-m} to
\[
Z_{\lambda,m} = \bigcup_{E \in \textnormal{Comp}(E_{\lambda,m})} \pi(E),
\]
where $\textnormal{Comp}(E_{\lambda,m})$ is the set of connected components of $E_{\lambda,m}$.
Because $E_{\lambda,m}$ is subanalytic due to Lemma~\ref{lem:Z-lambda-m}, so are its connected components; see Appendix~\ref{app:suban}\ref{app:suban-properties-1}.
Write $\pi_E$ for the restriction $\pi|_E$. 
The corresponding decomposition of $\mu_{\lambda,m}$ is
\[
\mu_{\lambda,m} = \sum_{E \in \textnormal{Comp}(E_{\lambda,m})}(\pi_E)_*\rho_{\lambda,m}.
\]
(Strictly speaking, $(\pi_E)_*\rho_{\lambda,m}$ refers to the pushforward of the \emph{restriction to $E$} of the measure $\rho_{\lambda,m} \cdot \Hcal^{d-m}$, but we permit this slight notational inaccuracy.)

We now observe that it suffices to fix any such $E$ and prove the lemma (without the last sentence) with $\pi(E)$ in place of $Z_{\lambda,m}$ and $(\pi_E)_* \rho_{\lambda,m}$ in place of $\mu_{\lambda,m}$. 
Indeed, suppose this has been done.
Then $Z_{\lambda,m}$ is the union of the subanalytic Whitney stratified subanalytic sets $\pi(E)$, and can be equipped with a common subanalytic Whitney refinement of their stratifications; see Appendix~\ref{app:suban}\ref{app:suban-whitney-refinement-2}.
Thanks to Appendix~\ref{app:suban}\ref{app:suban-whitney-exists}, since $Z_{\lambda,m}$ is locally closed, we may further suppose that the stratification has connected strata and satisfies the frontier condition.
Lastly, one simply notes that the desired behavior of $\mu_{\lambda,m}$ on individual strata is preserved when refining the stratifications.

We thus fix $E \in \textnormal{Comp}(E_{\lambda,m})$.
To construct the required stratification, we single out two ``bad'' subsets of $E$ that require separate treatment:

\textit{1. The rank deficiency locus.}
Consider the maximal rank of the derivative of $\pi_E$,
\[
r = \max_{p \in E} \rank(D\pi_E(p)),
\]
and let $J$ denote the $r$-dimensional Jacobian of $\pi_E$. Thus $J(p)$ is the product of the $r$ largest singular values of $D\pi_E(p)$.
More explicitly, letting $e_r(A)$ denote the sum of the $r \times r$ principal minors of the square matrix $A$, we have the formula
\[
J(p) = \sqrt{e_r(D\pi_E(p) D\pi_E(p)^*)}.
\]
The rank deficiency locus is
\[
\Sigma = \{p \in E \colon \rank(D\pi_E(p)) \le r-1\}.
\]
This equals the zero set in $E$ of the real-analytic function $J^2$, and is therefore a closed (possibly empty) semianalytic subset of $E$ of strictly lower dimension than $E$.

\textit{2. The density degeneracy locus.}
The density $\rho_{\lambda,m}$ is strictly positive and real-analytic everywhere in $E$ except possibly at points that intersect some subcritical component $E_\alpha$, $\alpha \in A_\subcr$, where it may degenerate (become zero or blow up).
We thus define the density degeneracy locus,
\[
T = E \cap \bigcup_{\alpha \in A_\subcr} E_\alpha.
\]
The simple normal crossing property of the $E_\alpha$ implies that $T$ is a finite union of real-analytic codimension-$1$ submanifolds of $E$.
In particular, $T$ is a closed subanalytic subset of $E$ of strictly lower dimension than $E$.

With $\Sigma$ and $T$ in hand, the ``good'' subset is obtained by removing these from $E$,
\begin{equation} \label{eq:lem-strat-n-good}
N = E \setminus (\Sigma \cup T).
\end{equation}
Since $\Sigma$ and $T$ are closed in $E$ and lower-dimensional, $N$ is a real-analytic submanifold of $M$ of dimension $\dim(E)$.
Both $\pi(E)$ and $\pi(N)$ are then subanalytic thanks to Appendix~\ref{app:suban}\ref{app:suban-images}, and both have dimension $r$. 
We claim that
\begin{equation} \label{eq:lem-strat-bad-set-low-dim}
\pi(E) \setminus \pi(N) \text{ has dimension at most $r-1$.}
\end{equation}
To see this, put $X = \pi(E)_\textnormal{reg} \setminus \pi(N)$. This is a subanalytic subset of the $r$-dimensional real-analytic manifold $\pi(E)_\textnormal{reg}$ thanks to Appendix~\ref{app:suban}\ref{app:suban-properties-1} and~\ref{app:suban-reg-sing}. 
Assume for contradiction that $\dim(X) = r$.
Then $X$ contains a nonempty set $O$ which is open in $\pi(E)_\textnormal{reg}$ and hence, by Appendix~\ref{app:suban}\ref{app:suban-reg-sing}, in $\pi(E)$.
By continuity, $\pi_E^{-1}(O)$ is open in $E$.
But we also have $\pi_E^{-1}(O) \subset \pi_E^{-1}(X) \subset \Sigma \cup T$, which has lower dimension than $E$ and so cannot contain any such open set.
This contradiction shows that $\dim(X) \le r-1$.
To deduce \eqref{eq:lem-strat-bad-set-low-dim}, it only remains to observe that $\dim(\pi(E)_\textnormal{sing}) \le r-1$; see again Appendix~\ref{app:suban}\ref{app:suban-reg-sing}.

We are now in a position to apply the coarea formula (Lemma~\ref{lem:coarea}) with $N$ as in \eqref{eq:lem-strat-n-good}, $\sigma = \pi|_N$, and $u = f \circ \pi_E \, \rho_{\lambda,m} / J$ for nonnegative measurable $f$.
Since $\pi_E = \sigma$ on $N$, the respective Jacobians coincide there, and $u = f(x) \, \rho_{\lambda,m} / J$ on $\sigma^{-1}(x)$.
The coarea formula then yields
\begin{equation} \label{eq:coarea-rho-lambda-m}
\int_N f \circ \pi_E \, \rho_{\lambda,m} \, d\Hcal^{\dim(E)}
= \int_{\pi(N)} f(x) \left( \int_{\sigma^{-1}(x)}  \frac{\rho_{\lambda,m}}{J} d\Hcal^{\dim(E)-r} \right) d\Hcal^r(x).
\end{equation}
(Recall from Remark~\ref{rem:m-in-euclidean} that $M$, and hence $N$, is embedded in an ambient Euclidean space, and this specifies $\Hcal^{\dim(E)}$.)
Because $\Sigma$ and $T$ are $\Hcal^{\dim(E)}$-nullsets, the left-hand side of \eqref{eq:coarea-rho-lambda-m} is unchanged if $N$ is replaced by $E$.
Noting also that $\sigma^{-1}(x) = \pi_E^{-1}(x) \setminus (\Sigma \cup T)$, and that $\pi(E) \setminus \pi(N)$ is a $\Hcal^r$-nullset due to \eqref{eq:lem-strat-bad-set-low-dim}, we obtain
\begin{align*}
\int_E f \circ \pi_E \, \rho_{\lambda,m} \, d\Hcal^{\dim(E)} = \int_{\pi(E)} f \,  \eta \, d\Hcal^r
\end{align*}
where we define
\begin{equation} \label{eq:density-eta}
\eta(x) = \int_{\pi_E^{-1}(x) \setminus (\Sigma \cup T)}  \frac{\rho_{\lambda,m}}{J} \, d\Hcal^{\dim(E)-r},  \quad x \in \pi(N),
\end{equation}
and arbitrarily set $\eta(x) = 0$ on the $\Hcal^r$-nullset $\pi(E) \setminus \pi(N)$.
This shows that the pushforward $(\pi_E)_*\rho_{\lambda,m}$ is absolutely continuous with respect to $\Hcal^r$ with density $\eta$.

We now exhibit a suitable subanalytic Whitney stratification of $\pi(E)$ and establish the required properties of $\eta$.
Thanks to Lemma~\ref{lem:fiber-integral-bounded-below},
\begin{equation} \label{eq:lem-strat-lower-bound}
\text{\parbox{.7\textwidth}{every $x_0 \in \pi(N)_\textnormal{reg}$ admits a neighborhood $W$ in $\R^d$ \\ and a constant $c > 0$ such that $\eta \ge c$ on $W \cap \pi(N)$.}}
\end{equation}
This motivates partitioning $\pi(E)$ into three parts,
\[
\pi(E) = \pi(N)_\textnormal{reg} \cup \pi(N)_\textnormal{sing} \cup ( \pi(E) \setminus \pi(N) ).
\]
Each of the three terms on the right-hand side is subanalytic, so we may select a subanalytic Whitney stratification of $\pi(E)$ such that each of the three terms is a union of strata; see Appendix~\ref{app:suban}\ref{app:suban-properties-1}, \ref{app:suban-reg-sing}, and \ref{app:suban-whitney-refinement-2}.
We know from \eqref{eq:lem-strat-bad-set-low-dim} and Appendix~\ref{app:suban}\ref{app:suban-reg-sing} that the second and third term are of dimension at most $r-1$.
Thus every $r$-dimensional stratum $S$ is contained in $\pi(N)_\textnormal{reg}$.
Thanks to \eqref{eq:lem-strat-lower-bound}, $\eta$ is locally bounded away from zero on $S$, so the Hamza condition \eqref{eq:hamza-condition} follows.
All other strata are $\Hcal^r$-nullsets, so $(\pi_E)_*\rho_{\lambda,m}$ is zero there.
This completes the proof.
\end{proof}

\appendix

\section{The coarea formula} \label{app:coarea}

The following version of the coarea formula can be deduced from \citet[Theorem~3.2.22]{MR257325}.
Recall that all manifolds are second countable and Hausdorff, and that submanifolds are understood to be embedded unless stated otherwise.

\begin{lemma}[Coarea formula] \label{lem:coarea}
Let $N$ be an $n$-dimensional $C^1$ submanifold of some ambient Euclidean space, let $\sigma \colon N \to \R^d$ be a $C^1$ map whose derivative has constant rank $r$, and let $u \colon N \to \R$ be nonnegative and Borel measurable. Then
\[
\int_N u \, J_\sigma \, d\Hcal^n = \int_{\sigma(N)} \left( \int_{\sigma^{-1}(x)} u \, d\Hcal^{n-r} \right) d\Hcal^r(x),
\]
where $J_\sigma(p)$ denotes the $r$-dimensional Jacobian of the derivative $D\sigma(p)$, i.e., the product of its $r$ nonzero singular values, and the inner integral on the right-hand side is
$\Hcal^r$-measurable.
\end{lemma}

\begin{remark}
The only role of the ambient Euclidean space is to make the Hausdorff measure well-specified. One could equally well consider a more general Riemannian manifold, possibly $N$ itself, in which case the Hausdorff measure is the one induced by the ambient volume measure.
\end{remark}

The next result establishes that the fiber integral in Lemma~\ref{lem:coarea} is locally bounded away from zero near regular points, assuming $u$ is positive and continuous.

\begin{lemma} \label{lem:fiber-integral-bounded-below}
Let $N$ and $\sigma$ be as in Lemma~\ref{lem:coarea} and let $u \colon N \to (0,\infty)$ be continuous. Define the fiber integral
\[
g(x) = \int_{\sigma^{-1}(x)} u \, d\Hcal^{n-r}, \quad x \in \sigma(N).
\]
Then every $x_0 \in \sigma(N)$ with the property that $\sigma(N)$ is an $r$-dimensional $C^1$ submanifold of $\R^d$ in some neighborhood of $x_0$ admits a neighborhood $W$ in $\R^d$ and a constant $c > 0$ such that $g \ge c$ on $W \cap \sigma(N)$.
\end{lemma}

\begin{proof}
Fix $x_0$ as in the statement and pick $p_0 \in \sigma^{-1}(x_0)$.
The constant rank theorem yields a chart $\varphi \colon U \to V_1 \times V_2$ of $N$ around $p_0$, where $V_1 \subset \R^r$ and $V_2 \subset \R^{n-r}$ are open disks, and a chart $\psi$ of $\R^d$ around $x_0$, such that $\psi \circ \sigma \circ \varphi^{-1}(y,z) = (y,0)$.
Consequently, $\Delta = \sigma(U) = \psi^{-1}(V_1 \times \{0\})$ is an embedded $r$-dimensional disk through $x_0$ contained in $\sigma(N)$, and we have $\sigma^{-1}(x) \cap U = \varphi^{-1}(\{y\} \times V_2)$ for any $x = \psi^{-1}(y,0) \in \Delta$.
We choose $U$ relatively compact in $N$. 
Then $\varphi$ is Lipschitz on $U$ with some constant $L$, so that
\[
\Hcal^{n-r}\big(\sigma^{-1}(x) \cap U\big) \ge \Hcal^{n-r}(V_2) / L^{n-r}, \quad x \in \Delta.
\]
Moreover, $u \ge c_0$ on $U$ for some constant $c_0 > 0$.
Since $u \ge 0$, it follows that $g \ge c_0 \Hcal^{n-r}(V_2) / L^{n-r}$ on $\Delta$.
Finally, by assumption there is a neighborhood $W_0$ of $x_0$ such that $R = W_0 \cap \sigma(N)$ is an $r$-dimensional $C^1$ submanifold of $\R^d$.
After shrinking $V_1$, and hence $U$ and $\Delta$, we may assume $\Delta \subset R$.
The $r$-dimensional submanifold $\Delta$ is then open in the $r$-dimensional manifold $R$, and therefore contains $W \cap R = W \cap \sigma(N)$ for some neighborhood $W \subset W_0$ of $x_0$.
The claim follows with $c = c_0 c_0 \Hcal^{n-r}(V_2) / L^{n-r} > 0$.
\end{proof}

\section{Subanalytic sets and stratifications} \label{app:suban}

Let $M$ be a real-analytic manifold and consider a subset $X \subset M$.
Following \cite{MR972342} we review the basic definitions and properties of subanalytic sets.

\begin{enumerate}
\item \label{app:def-semianalytic} $X$ is \emph{semianalytic} if it is locally determined by finitely many real-analytic functions in the following sense: every point in $M$ has a neighborhood $U$ such that $X \cap U = \bigcup_{i=1}^p \bigcap_{j=1}^q \{f_{ij} \ \varepsilon_{ij} \ 0\}$, where $f_{ij}$ is a real-analytic function on $U$ and $\varepsilon_{ij}$ is either ``$>$'' or ``$=$''.
In particular, every closed real-analytic submanifold of $M$ is semianalytic.

\item \label{app:def-subanalytic} $X$ is \emph{subanalytic} if each point in $M$ has a neighborhood $U$ such that $X \cap U = \pi(A)$, where $A$ is a relatively compact semianalytic subset of $M \times N$ for some real-analytic manifold $N$, and $\pi \colon M \times N \to M$ is the projection map. In particular, every semianalytic set is subanalytic.

\item Let $N$ be a real-analytic manifold. A map $f \colon X \to N$ is \emph{subanalytic} if its graph is subanalytic in $M \times N$.
In particular, every real-analytic map is subanalytic.
\end{enumerate}

Here are some fundamental properties of subanalytic sets and maps.

\begin{enumerate}[resume]
\item \label{app:suban-properties-1} Boolean combinations (complements, finite unions, finite intersections), closures, interiors, and connected components of subanalytic sets are subanalytic, and the family of connected components is locally finite. See the paragraphs following Definition~3.1, as well as Theorem~3.10, of \citet{MR972342}.

\item \label{app:suban-images} Images of relatively compact subanalytic sets under subanalytic maps are subanalytic. See \citet{MR972342}, following Definition~3.2. As a consequence, the image of an arbitrary subanalytic set under a proper subanalytic map is subanalytic.
Further, the inverse image of an arbitrary subanalytic set under a continuous subanalytic map is subanalytic.
\end{enumerate}

A point $p \in X$ is \emph{smooth of dimension $k$} if in some neighborhood of $p$ in $M$, $X$ is a $k$-dimensional analytic submanifold of $M$. The dimension of $X$, $\dim(X)$, is the highest dimension of its smooth points. 
The \emph{regular locus} $X_\textnormal{reg}$ is the set of $\dim(X)$-dimensional smooth points (in particular itself a $\dim(X)$-dimensional analytic submanifold), and the \emph{singular locus} is $X_\textnormal{sing} = X \setminus X_\textnormal{reg}$.
See \citet{MR972342}, Definition~3.3, Remark~3.5, and Definition~7.1.

\begin{enumerate}[resume]
\item \label{app:suban-reg-sing} If $X$ is subanalytic, $X_\textnormal{sing}$ is subanalytic and closed in $X$. As a consequence, $X_\textnormal{reg}$ is also subanalytic and open in $X$. 
See \citet{MR972342}, Theorem~7.2.
\end{enumerate}

Subanalytic sets in $\R^d$ have the important property that they admit \emph{Whitney stratifications}: decompositions into smooth submanifolds of $\R^d$ that fit together in a controlled way.
Following \cite{MR1463945} we review the basic definitions and properties.

Let $X \subset \R^d$.
A \emph{stratification} $\Scal$ of $X$ is a partition of $X$ into $C^\infty$ submanifolds of $\R^d$ called \emph{strata}, which is locally finite in the sense that every point of $X$ has a neighborhood that only meets finitely many strata.
The \emph{frontier condition} states that
\[
\text{if $R,S \in \Scal$, $R \ne S$, and $R\cap \overline{S} \ne \emptyset$, then $R \subset \overline S \setminus S$.}
\]
A \emph{Whitney stratification} is one which satisfies the \emph{Whitney condition}, also known as Whitney's \emph{condition (b)}; see \cite{MR2958928}. This condition states that for any pair of distinct strata $R,S$, points $x_k \in R$, $y_k, y \in S$, and linear subspaces $\ell,\tau \subset \R^d$ of dimension one and $\dim(R)$ respectively,
\[
\text{if $x_k \to y$, $y_k \to y$, $\overline{x_k y_k} \to \ell$, and $T_{x_k}R \to \tau$, then $\ell \subset \tau$.}
\]
Here $\overline{x_k y_k}$ is the one-dimensional subspace (line) going through zero and $y_k-x_k$, $T_{x_k}R$ is the tangent space of $R$ at $x_k$, and convergence of subspaces is understood in the Grassmannian topology. See $\S$I.1 of \cite{MR1463945}.

We now review some Whitney stratification theory for subanalytic sets $X \subset \R^d$. A stratification $\Scal$ of such a set is called \emph{subanalytic} if it is locally finite at every point of $\R^d$ (not just of $X$) and each stratum is an analytic (not just $C^\infty$) submanifold as well as a subanalytic subset of $\R^d$.

\begin{enumerate}[resume]
\item \label{app:suban-whitney-exists} Any subanalytic set $X \subset \R^d$ admits a subanalytic Whitney stratification $\Scal$. See \cite{MR1463945}, Lemma~I.2.2. If $X$ is locally closed, the refined stratification $\Scal' = \bigcup_{S \in \Scal} \textnormal{Comp}(S)$ consisting of the connected components of the strata in $\Scal$ is a subanalytic Whitney stratification with connected strata satisfying the frontier condition; see Remark~\ref{rem:frontier-condition} below.
\end{enumerate}

\begin{remark} \label{rem:frontier-condition}
The frontier condition is sometimes included as part of the definition of Whitney stratification, see e.g.\ \cite{MR2958928}, Section~5. \cite{MR1463945} does not do this, so neither do we. However, \cite{MR1463945} does show that any Whitney stratification $\Scal$ (as defined above) of a locally closed set in $\R^d$ with connected strata automatically satisfies the frontier condition; see the comment on p.~24 of \cite{MR1463945} right after the proof of Lemma~I.1.11.
If the strata of $\Scal$ are not connected, but its family of connected components $\Scal' = \bigcup_{S \in \Scal} \textnormal{Comp}(S)$ is locally finite, then $\Scal'$ is a stratification, has connected strata, and satisfies the Whitney property if $\Scal$ does.
Thanks to \ref{app:suban-whitney-exists} and \ref{app:suban-properties-1}, this yields in particular that any locally closed subanalytic set in $\R^d$ admits a Whitney stratification with connected strata satisfying the frontier condition.
\end{remark}

Given partitions $\Rcal, \Scal$ of two (possibly different) sets, $\Scal$ is called a \emph{refinement} of $\Rcal$ if every set in $\Rcal$ is a union of sets in $\Scal$.
If $\Scal$ is a Whitney stratification, we call it a \emph{Whitney refinement}.
The following results are consequences of \citet[Lemma~I.2.2]{MR1463945}. The statements have certain mutual redundancies, but are stated in this way for ease of reference.

\begin{enumerate}[resume]
\item \label{app:suban-whitney-refinement} 
If $\Scal_1,\Scal_2$ are subanalytic Whitney stratifications of two subanalytic sets $X_1,X_2$ in $\R^d$, then the intersection partition $\{S_1 \cap S_2 \colon S_1 \in \Scal_1, \, S_2 \in \Scal_2\}$ of $X_1 \cap X_2$ admits a subanalytic Whitney refinement.

\item \label{app:suban-whitney-refinement-2} Any locally finite family $\{X_i\}$ of subanalytic Whitney stratified subanalytic sets in $\R^d$ admits a common subanalytic Whitney refinement $\Scal$ which stratifies the union $X = \bigcup_i X_i$.

\end{enumerate}

Fix an open domain $D \subset \R^d$ and consider a subset $X \subset D$ equipped with a stratification $\Scal$.
Define the space of $C^1$ functions with \emph{stratified gradient} as
\[
C^1_{\textnormal{strat}}(\Scal) = \{f|_X \colon f \in C^1(D) \text{ and } \nabla f(x) \in T_x S \text{ for all } S \in \Scal, \ x \in S\},
\]
and let $C^1_{c,\,\textnormal{strat}}(\Scal)$ denote the subspace of functions with compact support in $D$.
Although not made explicit in the notation, these definitions do depend on $D$.
The following result gathers some properties of $C^1_{c,\,\textnormal{strat}}(\Scal)$ that are crucial for its use as the domain of a (pre-)Dirichlet form.

\begin{lemma} \label{lem:c-strat}
Let $D \subset \R^d$ be an open domain, and let $X \subset D$ be a subanalytic set closed in $D$ with a subanalytic Whitney stratification $\Scal$ that satisfies the frontier condition.
We then have:
\begin{enumerate}
\item\label{lem:c-strat-1} $C^1_{c,\,\textnormal{strat}}(\Scal)$ is a subalgebra of $C_c(X)$, i.e., a linear subspace closed under pointwise products;
\item\label{lem:c-strat-2} for every $\varepsilon > 0$, there is a function $\phi_\varepsilon \colon \R \to \R$ satisfying
\[
\phi_\varepsilon(t) = t \text{ on } [0,1], \quad -\varepsilon \le \phi_\varepsilon \le 1 + \varepsilon, \quad 0 \le \phi_\varepsilon(t) - \phi_\varepsilon(s) \le t-s \text{ for } s \le t,
\]
such that $\phi_\varepsilon \circ f \in C^1_{c,\,\textnormal{strat}}(\Scal)$ for every $f \in C^1_{c,\,\textnormal{strat}}(\Scal)$.
\item\label{lem:c-strat-3} for any open relatively compact set $G \subset X$ and compact set $K \subset G$, there is a function $h \in C^1_{c,\,\textnormal{strat}}(\Scal)$ with $h \ge 0$, $h = 1$ on $K$, and $h = 0$ on $X \setminus G$;
\item\label{lem:c-strat-4} $C^1_{c,\,\textnormal{strat}}(\Scal)$ is dense in $C_c(X)$ with respect to the uniform norm, and dense in $L^2(X,\mu)$ for any Radon measure $\mu$ on $X$.
\end{enumerate} 
\end{lemma}

\begin{proof}
\ref{lem:c-strat-1}: Immediate from linearity of the gradient and the product rule.
In particular, $\nabla(fg)(x) = f(x)\nabla g(x) + g(x) \nabla f(x)$ belongs to $T_x S$ provided both $\nabla f(x)$ and $\nabla g(x)$ do.

\ref{lem:c-strat-2}: Immediate from the chain rule in that \emph{any} $C^1$ function $\phi_\varepsilon$ with the displayed properties will ensure that $\phi_\varepsilon \circ f \in C^1_{c,\,\textnormal{strat}}(\Scal)$ for every $f \in C^1_{c,\,\textnormal{strat}}(\Scal)$. Many such $\phi_\varepsilon$ exist.

\ref{lem:c-strat-3}: 
Because $G$ is open and relatively compact in $X$, it is of the form $G = U \cap X$ for a set $U$ which is open and relatively compact in $D$.
Let $F$ be a compact subanalytic subset of $\R^d$ such that $U \subset F \subset D$ (take, for example, a finite union of closed balls), and equip it with a subanalytic Whitney stratification $\Rcal$.
We then apply \citet[Theorem~4.2]{MR3271275}. 
Using their notation, we take $Q = X \cap F$, $\Acal$ a subanalytic Whitney refinement of the intersection partition of $X \cap F$ satisfying the frontier condition (see Appendix~\ref{app:suban}\ref{app:suban-whitney-exists} and~\ref{app:suban-whitney-refinement} above), a continuous function $f \colon \R^d \to \R$ with $f = 1$ on a neighborhood of $K$ and $\supp(f) \subset U$, $\varepsilon = 1/4$, and a neighborhood $V_1$ of $\supp(f)$ with $V_1 \subset U$. 
Their theorem then yields a function $g \in C^1_{c,\,\textnormal{strat}}(\Acal)$ with $\supp(g) \subset V_1$ and $|f(x) - g(x)| < 1/4$ for all $x \in \R^d$.
In fact, $g$ belongs to $C^1_{c,\,\textnormal{strat}}(\Scal)$.
Indeed, the gradient $\nabla g$ vanishes on $X \setminus F$, so is trivially stratified there.
For $x \in X \cap F$, let $A \in \Acal$ and $S \in \Scal$ be the strata that contain $x$. 
Since $A \subset S$ and $\nabla g(x)$ is tangent to $A$, it is also tangent to $S$.
We now pick a nonnegative $\phi \in C^1(\R)$ with $\phi(t) = 0$ for $t \le 1/4$ and $\phi(t) = 1$ for $t \ge 3/4$.
The function $h = \phi \circ g$ has the required properties.

Note that \citet{MR3271275} assume that $Q$ is closed and $\Acal$ finite. In our situation $Q$ is compact, and thus has finitely many strata by local finiteness of the stratification $\Acal$.
Moreover, \citet{MR3271275} include the frontier condition as part of the definition of a stratification. Since we do not, we have to argue this separately.

\ref{lem:c-strat-4}: Density in $C_c(X)$ follows from the locally compact version of the Stone--Weierstrass theorem; see \citet[Theorem~4.52]{MR1681462}. Indeed, $C^1_{c,\,\textnormal{strat}}(\Scal)$ is a subalgebra of $C_0(X)$ thanks to \ref{lem:c-strat-1}, and \ref{lem:c-strat-3} ensures that it separates points and vanishes nowhere.
Density of $C_c(X)$ in any $L^2(X,\mu)$ with $\mu$ a Radon measure on $X$ is a standard fact from measure theory.
This completes the proof.
\end{proof}

\begin{remark} \label{rem:dru-lar}
The assumptions of Lemma~\ref{lem:c-strat} are not the weakest possible.
For instance, the application of \citet{MR3271275} only requires $C^2$ regularity of the strata and Whitney's condition (a), which is weaker than condition (b).
Moreover, the subanalyticity hypotheses are only used to make sure the intersection $X \cap F$ has a Whitney refinement satisfying the frontier condition.
A local version of \citet[Theorem~4.2]{MR3271275} would remove the need to intersect with $F$, and thus the need to assume subanalyticity.
We do not pursue such improvement here, however.
\end{remark}

\section{Dirichlet forms} \label{app:dirichlet}

Our main reference for the theory of Dirichlet forms and their connection to Markov processes is \citet{MR2778606}.

Let $X$ be a locally compact Hausdorff space and $\mu$ a Radon measure on $X$ with full support.
A \emph{Dirichlet form} $\Ecal$ on $L^2(X,\mu)$ is a closed Markovian nonnegative definite symmetric bilinear form on a dense domain $\Fcal \subset L^2(X,m)$.
Closed means that $\Fcal$ is complete under the squared norm $\Ecal_1(f,f) = \Ecal(f,f) + \|f\|_{L^2(X,\mu)}^2$. Markovian means that if $f \in \Fcal$ then $g = (0 \vee f) \wedge 1 \in \Fcal$ and $\Ecal(g,g) \le \Ecal(f,f)$.
A Dirichlet form is \emph{regular} if $\Fcal \cap C_c(X)$ is $\Ecal_1$-dense in $\Fcal$ and uniformly dense in $C_c(X)$.
It is \emph{strongly local} if $\Ecal(f,g) = 0$ whenever $f,g \in \Fcal$ have compact supports and $f$ is constant on a neighborhood of the support of $g$.

There is a one-to-one correspondence between strongly local regular Dirichlet forms and Markov processes with continuous trajectories in the one-point compactification of $X$; see \citet[Theorems~4.2.8, 4.5.3 and~7.2.1]{MR2778606}.
This is the reason we are interested in Dirichlet forms in the first place.

In applications it is common to start with a symmetric bilinear form $\Ecal$ on a dense domain $\Dcal \subset L^2(X,m)$, not yet a Dirichlet form, and pass to the \emph{closure} of $(\Ecal,\Dcal)$.
The closure is the smallest closed extension of $(\Ecal,\Dcal)$, which exists precisely when the form is \emph{closable}. 
This means that whenever a sequence of elements $f_n \in \Dcal$ satisfies $\Ecal(f_n-f_m,f_n-f_m) \to 0$ and $\|f_n\|_{L^2(X,\mu)} \to 0$  as $m,n \to \infty$, one has $\Ecal(f_n,f_n) \to 0$.
\citet{Hamza1975} introduced a simple closability criterion for forms of the type $\int_\R f' \, g' \, d\mu$, namely that $\mu$ have a Lebesgue density $\rho > 0$ with $1/\rho$ locally integrable.
This condition has been generalized in various directions, e.g.\ to the case where $X$ is a Riemannian manifold; see \citet[Theorem~4.2]{MR1037315}.
We use the same basic argument to derive a closability condition when $X$ is a certain subanalytic set.

\begin{lemma} \label{lem:closability-stratified}
Let $D \subset \R^d$ be an open domain, and let $X \subset D$ be a subanalytic set closed in $D$ with a subanalytic Whitney stratification $\Scal$ that satisfies the frontier condition.
Let $\mu$ be a Radon measure on $X$ with full support such that the restriction $\mu|_S$ to each stratum $S \in \Scal$ is either zero or satisfies the Hamza condition:
\begin{equation} \label{eq:hamza-condition}
\mu|_S \ll \Hcal^{\dim(S)} \text{ with density $\rho_S > 0$ such that } \frac{1}{\rho_S} \in L^1_\textnormal{loc}(S, \Hcal^{\dim(S)}).
\end{equation}
Then the bilinear form on $L^2(X, \mu)$ given by
\[
\Ecal(f,g) = \int_X \nabla f \cdot \nabla g \, d\mu, \quad f,g \in \Dcal = C^1_{c,\,\textnormal{strat}}(\Scal),
\]
is closable, and its closure is a strongly local regular Dirichlet form which has $\Dcal$ as a special standard core. 
\end{lemma}

\begin{proof}
For each $S \in \Scal$, consider the nonnegative symmetric bilinear form $\Ecal_S$ with domain $\Dcal_S \subset L^2(S, \mu|_S)$ given by
\[
\Ecal_S(f,g) = \int_S \nabla_S f \cdot \nabla_S g \, d\mu|_S, \quad \Dcal_S = \{h|_S \colon h \in C^1_c(D)\}.
\]
This is well-defined since $X$ is closed in $D$, which ensures that the supports of the elements of $\Dcal_S$ are compact in $X$ and thus have finite $\mu$-measure.
Suppose for the moment we know $(\Ecal_S, \Dcal_S)$ is closable.
Then so is the direct sum form on $H_\oplus = \bigoplus_{S \in \Scal} L^2(S, \mu|_S)$,
\[
\Ecal_\oplus = \bigoplus_{S \in \Scal} \Ecal_S, \quad \Dcal_\oplus = \bigoplus_{S \in \Scal} \Dcal_S.
\]
Indeed, the direct sum of any family of closable symmetric nonnegative definite bilinear forms is again closable; this is implied, for instance, by \citet[Theorem~1.1]{MR1037315}.
Now, $H_\oplus$ can be identified with $L^2(X,\mu)$, $\Dcal$ with a linear subspace of $\Dcal_\oplus$, and $\Ecal$ with the restriction of $\Ecal_\oplus$ to this linear subspace.
The restriction of a closable form remains closable, so we conclude that $(\Ecal,\Dcal)$ is closable.
That the closure is a strongly local regular Dirichlet form with $\Dcal$ as a special standard core follows from \citet[Theorem~3.1.2 and Exercise~3.1.1]{MR2778606}.
The hypotheses of those results follow from Lemma~\ref{lem:c-strat}, noting in particular that for any function $\phi_\varepsilon$ as in Lemma~\ref{lem:c-strat}\ref{lem:c-strat-2} one has $\Ecal(\phi_\varepsilon \circ f, \phi_\varepsilon \circ f) \le \Ecal(f,f)$ for all $f \in \Dcal$.

It remains to show that each $(\Ecal_S, \Dcal_S)$ is closable.
We may suppose that $\mu|_S$ satisfies the Hamza condition \eqref{eq:hamza-condition}, because otherwise $\Ecal_S$ is zero and trivially closable.
Write $k = \dim(S)$.
Following \citet[Theorem~4.2]{MR1037315}, note that for any nonnegative $f \in L^2(S,\mu|_S)$ and $g \in C_c(S)$, the Cauchy--Schwarz inequality yields
\begin{equation} \label{eq:cauchy-schwartz-hamza}
\bigg( \int_S f \, g \, d\Hcal^k \bigg)^2 \le \bigg( \int_S f^2 \rho_S \, d\Hcal^k \bigg) \bigg( \int_S g^2 \rho_S^{-1} \, d\Hcal^k \bigg),
\end{equation}
where the second factor on the right-hand side is finite thanks to \eqref{eq:hamza-condition}.
Pick $f_n \in \Dcal_S$ such that $\Ecal_S(f_n-f_m,f_n-f_m) \to 0$ and $\|f_n\|_{L^2(S,\mu|_S)} \to 0$.
Then the sequence of gradients $\nabla_S f_n$ is Cauchy in $L^2(S,\mu|_S)$ and converges in norm to some tangent vector field $F$.
For any compactly supported smooth tangent vector field $G$ on $S$, integration by parts yields
\begin{equation} \label{eq:ibp-hamza}
\int_S \nabla_S f_n \cdot G \, d\Hcal^k = - \int_S f_n \, \textnormal{div}_S(G) \, d\Hcal^k.
\end{equation}
Using \eqref{eq:cauchy-schwartz-hamza} we deduce that the left-hand side of \eqref{eq:ibp-hamza} converges to $\int_S F \cdot G \, d\Hcal^k$, and that the right-hand side converges to zero.
Since $G$ was arbitrary, $F = 0$.
In other words, $\Ecal_S(f_n,f_n) \to 0$, as required for closability.
\end{proof}

\bibliography{bibliography}
\bibliographystyle{plainnat}

\end{document}